\documentclass[11pt,reqno]{amsart}
\usepackage{hyperref}
\usepackage[top=1.2in, bottom=1.2in, left=1.2in, right=1.2in]{geometry}
\usepackage{amsfonts, amssymb, amsmath, mathtools, dsfont}%, enumerate, bm, xspace, graphicx, caption, subcaption}
\usepackage{enumerate}

\usepackage{graphicx}

\usepackage{subcaption} 
\usepackage[font=small]{caption} 
\numberwithin{equation}{section}

\DeclareMathOperator{\F}{\mathbb{F}}

\DeclareMathOperator{\sgn}{sgn}

\DeclareMathOperator{\Span}{span}

\renewcommand{\Pr}[2][]{\mathbb{P}_{#1} \left\{ #2 \rule{0mm}{3mm}\right\}}

\newcommand{\ip}[2]{\langle#1,#2\rangle}

\def \N {\mathbb{N}}

\def \R {\mathbb{R}}

\def \XX {\mathcal{X}}

\def \a {\alpha}

\def \Plong {P_{\mathrm{long}}}
\def \Pshort {P_{\mathrm{short}}}

\newtheorem{theorem}{Theorem}[section]

\newtheorem{corollary}[theorem]{Corollary}
\newtheorem{lemma}[theorem]{Lemma}
\newtheorem{conjecture}[theorem]{Conjecture}

\newtheorem{definition}[theorem]{Definition}

\theoremstyle{remark}

\begin{document}

\title{Polynomial threshold functions, hyperplane arrangements, and random tensors}

\author{Pierre Baldi \and Roman Vershynin}
\date{\today}

\address{Department of Computer Science, University of California, Irvine}
\email{pfbaldi@uci.edu}

\address{Department of Mathematics, University of California, Irvine}
\email{rvershyn@uci.edu}

\thanks{Work in part supported by DARPA grant D17AP00002 to P. B. and U.S. Air Force grant FA9550-18-1-0031 to R. V}

\begin{abstract}
A simple way to generate a Boolean function is to take the sign of a real polynomial in $n$ variables. Such Boolean functions are called polynomial threshold functions. How many low-degree polynomial threshold functions are there? The partial case of this problem for degree $d=1$ was solved by Zuev in 1989, who showed that the number $T(n,1)$ of linear threshold functions satisfies $\log_2 T(n,1) \approx n^2$, up to smaller order terms. However the number of polynomial threshold functions
for any higher degrees, including $d=2$, has remained open. We settle this problem for all fixed degrees $d \ge1$, showing that $ \log_2 T(n,d) \approx n \binom{n}{\le d}$. The solution relies on connections between the theory of Boolean threshold functions, hyperplane arrangements, and random tensors. Perhaps surprisingly, it uses also a recent result of E.~Abbe, A.~Shpilka, and A.~Wigderson on Reed-Muller codes.
\end{abstract}

\maketitle

\setcounter{tocdepth}{1}
\tableofcontents

\section{Introduction}
%===============
\subsection{The problem}			\label{s: problem}
Neural networks and deep learning models and applications \cite{schmidhuber2015deep}
rely on a simplified neuronal model that goes back at least to the work of
W.~McCulloch and W.~Pitts in the 1940s \cite{mcculloch:43}. In this model, a neuron is viewed as a processing unit which, given $n$ inputs $x=(x_1,\ldots,x_n)$, produces an output of the form
$y=f(s)=f(\sum w_ix_i + t)$ where 
the coefficients $w_i$ represent the synaptic weights, $t$ is a threshold,
the weighted average $s$ is the activation, and $f$ is the transfer function. When the inputs are restricted to the Boolean cube $\{-1,1\}^n$ and $f$ is the sign function ($f=\sgn$), the neuron operates as a Boolean linear threshold function.

In search for both more powerful computational models, as well as greater biological realism that may take into account non-linear interactions between synapses along neuronal dendritic trees,
it is natural to replace the linear activation with a polynomial activation
$p(x)$ of degree $d$ so that $y=f(p(x))$
\cite{Baldi}. Again, considering the Boolean case, a Boolean function $f : \{-1,1\}^n \to \{-1,1\}$ is called a {\em polynomial threshold function} 
if it has the form: 
$$
f(x) = \sgn(p(x))
$$ 
for some real-valued polynomial $p : \R^n \to \R$ that has no roots in the Boolean cube
$\{-1,1\}^n$. 
Thus, $f$ takes the value $1$ at any point where the polynomial $p$ is positive, 
and $-1$ at any point where the polynomial $p$ is negative.
Up to a factor of two due to the $\sgn$ operation, polynomial threshold functions can be identified with partitions
of the Boolean cube $\{-1,1\}^n$ into two classes by polynomial surfaces corresponding to  $p(x)=0$. Figure~\ref{fig: partitions} illustrates the partitions of the two-dimensional cube obtained by linear and quadratic threshold functions, associated with polynomials of degree one and two respectively. {\it The main goal in this paper is to estimate the number of different polynomial threshold functions of degree $d$.}

\begin{figure}[htp]			
  \centering 
  \includegraphics[width=.6\textwidth]{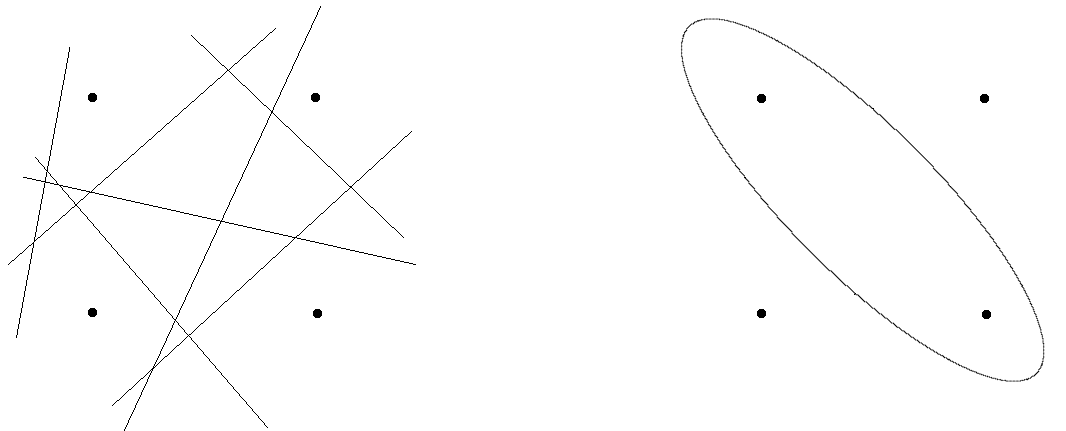} 
  \caption{On the left, all $7$ possible linear partitions of the Boolean square $\{-1,1\}^2$ are shown. 
  These define $14$ Boolean linear threshold functions 
  of two variables. On the right, using quadratic surfaces (ellipses), one can realize on additional partition, for a total of $8$ possible quadratic partitions. 
  These partitions are associated with $2 \cdot 8 = 16$ Boolean quadratic threshold functions 
  of two variables corresponding, in this case, to all possible Boolean functions of two variables.}
  \label{fig: partitions}
\end{figure}

\subsection{General background}			\label{s: background}
As previously indicated, the study of linear and polynomial threshold functions is implicit in some of the first models of neural activity by W.~McCulloch and W.~Pitts in the 1940s \cite{mcculloch:43}. Linear threshold functions were studied by T. Cover 
\cite{cover1965geometrical}, S. Muroga \cite{muroga1965lower},
M. Minsky and S. Papert in their book on perceptrons
\cite{Minsky-Papert}, and others in the 1960s. Since then, 
linear and polynomial threshold functions have been extensively used and studied in 
%Any Boolean function is a polynomial threshold function of degree at most $n$ \cite{ODonnell}, 
%and for most Boolean functions the lowest degree of the polynomial is approximately $n/2$. 
complexity theory, machine learning, and network theory; see, for instance, \cite{baldi87,Baldi,baldi88b,Bruck, Saks, Beigel, Krause-Pudlak, Sherstov, ABFR, BRS, ACW, Klivans-ODonnell-Servedio, Klivans-Servedio, DOSW, OS1, OS2, BGS, Kane}. An introduction to polynomial threshold functions
can be found in \cite[Chapter~5]{ODonnell book}, \cite[Chapter~4]{Anthony}, and \cite{Saks}. Linear and polynomial threshold functions remain a fundamental model for biological or neuromorphic neurons and, together with their continuous approximations, are at the center of all the current developments and applications of deep learning \cite{schmidhuber2015deep,baldireview2018}.

\subsection{The notion of capacity}	\label{s: capacity}

The standard description of the neuronal model given above focuses on the processing aspect of the neuron, its input-output function. However, there is a second, at least equally important aspect, which is the storage aspect. As described, a neuron is also a storage device capable of storing information in the linear or polynomial synaptic weights of its activation function. Thus it is natural to ask how many bits can be stored in a neuron. This question is intimately related to the recently-introduced 
notion of capacity
\cite{baldi2018neuronal,baldi2019capacity}  .

Given a class of functions $A$, such as all the functions that can be implemented by a neuron--or a network of neurons--as the synaptic weights are varied, we define the capacity of $C(A)$ as the binary logarithm $C(A)=\log_2 \vert A\vert$. In the discrete case, the only one to be considered here, $\vert A \vert $ is simply the number of functions in $A$ (in the continuous case, one must define a notion of volume). Thus the capacity is the number of bits required to specify an element of $A$. Remarkably, when $A$ is associated with a neural architecture comprising one or many neurons, the capacity can be viewed as the number of effective bits that can be extracted from a training set and stored in the neural architecture\cite{baldi2019capacity}. Estimating the capacity of a single linear threshold function has a long history reviewed below and, recently, we were able to estimate the capacity of networks of linear threshold function. Thus, the primary goal here is to begin extending these results beyond the linear case by estimating the capacity of a single polynomial threshold function.

\subsection{The capacity of linear threshold functions}	\label{s: linear}

There are $2^{2^n}$ Boolean functions of $n$ variables. Let $T(n,d)$ denote the number of polynomial threshold functions of fixed degree $d$. 
The asymptotical behavior of $T(n,d)$ has been known only for the linear case $d=1$. 
The work of T. Cover 
\cite{cover1965geometrical} and others used a simple hyperplane counting argument to show that $T(n,1)$ is upperbounded by $2^{n^2}$. Recursive constructions by S. Muroga \cite{muroga1965lower} and others provided lower bounds of the form $2^{\alpha n^2}$ with values of $\alpha$ that were significantly below 1.  

Yu.~Zuev \cite{Zuev Doklady, Zuev} was finally able to show in 1989 that the upper bound 
$2^{n^2}$ is asymptotically tight: the number of linear threshold functions satisfies:
\begin{equation}	\label{eq: Zuev}
\Big(1 - \frac{10}{\log n} \Big) \cdot n^2
\le \log_2 T(n,1) 
\le n^2.
\end{equation}
J.~Kahn, J.~Koml\'os, E.~Szemer\'edi \cite[Section~4]{KKS}
further improved this result to:
$$
\log_2 T(n,1) = n^2 - n \log_2 n \pm O(n).
$$
Although Zuev's breakthrough led to some progress in understanding linear threshold functions (see e.g. \cite{Ojha, Zunic, DSTW, Irmatov, Irmatov2, KV}), the same problem for higher degrees has remained open. M.~Saks explicitly asked about the asymptotical behavior of $T(n,d)$ in 1993 \cite[Problem~2.35]{Saks}. %whether the limit of $\log_2 T(n,d) / n^{d+1}$ exists as $n \to \infty$. 
Even the asymptotic behavior of the number of quadratic threshold functions $T(n,2)$ has so far remained unknown. 

\subsection{The capacity of polynomial threshold functions: the main result}	\label{s: polynomial}

Each of the $2^{2^n}$ Boolean functions of $n$ variables can be expressed as a polynomial of degree at most $n$: to see this, write the function $f$ in conjunctive (or disjunctive) normal form, or take the Fourier transform of $f$. In particular, every Boolean function $f$ is a polynomial threshold function, but the polynomial that represents $f$ often has high degree. A conjecture of J.~Aspnes {\em et al.} \cite{ABFR} and C.~Wang and A.~Williams \cite{Wang-Williams} states that, for most Boolean functions $f(x)$, the lowest degree of $p(x)$ such that $f(x) = \sgn(p(x))$ is either $\lfloor n/2 \rfloor$ or $\lceil n/2 \rceil$. M.~Anthony \cite{Anthony classification} and independently N.~Alon (see \cite{Saks}) proved one half of this conjecture, showing that for most Boolean functions the lower degree of $p(x)$ is at least $\lceil n/2 \rceil$. The other half of the conjecture was settled, in an approximate sense, by R.~O'Donnell and R.~A.~Servedio \cite{OS1} who gave an upper bound $n/2 + O(\sqrt{n \log n})$ on the degree of $p(x)$. 

While low degree polynomial threshold functions may be relatively rare within the space of Boolean functions, they are of particular interest both theoretically and practically, due to their functional simplicity and their utilization in biological modeling and neural network applications. Thus the most fundamental open question regarding low-degree polynomial threshold functions is: 
\begin{quote}
{\em How many low-degree polynomial threshold functions are there?
  Equivalently, how many different ways are there to partition the Boolean cube 
  by polynomial surfaces of low degree?}
  Equivalently how many bits can effectively be stored in the coefficients of a polynomial threshold function? 
\end{quote}
In the following theorem, we settle the problem for all fixed degrees $d \in \N$.

\begin{theorem}			\label{thm: main}
  For any positive integers $n$ and $d$ such that $1 \le d \le n^{0.9}$, 
  the number of Boolean polynomial threshold functions $T(n,d)$ satisfies\footnote{Here and
  	in the rest of the paper, $\binom{n}{\le d}$ denotes the binomial sum up to term $d$, 
	i.e. $\binom{n}{\le d} \coloneqq \binom{n}{0} + \binom{n}{1} + \cdots + \binom{n}{d}$.}
  $$
  \Big( 1 - \frac{C}{\log n} \Big)^d \cdot n \binom{n}{\le d}
  \le \log_2 T(n,d) 
  \le n \binom{n}{\le d}.
  $$ 
\end{theorem}

In this theorem and the rest of the paper, $C$ denotes a positive absolute constant;
its value does not depend on $n$ or $d$. The exact value of $C$ may be different in 
different parts of this paper.
For linear threshold functions, i.e. for $d=1$, Theorem~\ref{thm: main} yields Zuev's result \eqref{eq: Zuev} up to the absolute constant $C$.

The upper bound in Theorem~\ref{thm: main} holds for all $1 \le d \le n$; 
this bound is known and can be derived from
counting regions in hyperplane arrangements; we reprove it in Section~\ref{s: linearization}
for completeness. 
The lower bound in Theorem~\ref{thm: main} is new. 
It will be clear from the argument that the exponent $0.9$ in the constraint on $d$ 
can be replaced by any constant strictly less than $1$ at the cost of changing the absolute constant $C$. 

\iffalse
In \cite{baldi2018neuronal,baldi2019capacity}, we define the capacity $C(A)$ of a finite class of functions $A$ by:
$C(A)=\log_2 \vert A \vert$, and this paper precisely deals with the capacity of polynomial threshold functions. The capacity is thus the number of bits required to communicate an element of $A$. Remarkably, in the case of polynomial threshold functions, it can also be interpreted as the number of bits that can be stored in such as function during learning, i.e. as its coefficients are varied. 
\fi

For small degrees $d$, namely for $d = o(\log n)$, the factor $(1-C/\log n)^d$ becomes 
$1-o(1)$ and Theorem~\ref{thm: main} yields in this case the asymptotically tight bound on the capacity:
$$
\log_2 T(n,d) = \left( 1-o(1) \right) \cdot n \binom{n}{\le d}.
%= \left( 1-o(1) \right) \frac{n^{d+1}}{d!}.
$$
To better understand this bound, note that a general polynomial 
of degree $d$ has $\binom{n}{\le d}$ monomial terms. 
Thus, Theorem~\ref{thm: main} yields the following result on communication complexity, and learning: 
\begin{quote}
  {\em To communicate a polynomial threshold function, one needs to spend approximately $n$ bits per monomial term. During learning, approximately $n$ bits can be stored per monomial term.}
\end{quote}

In some situations, it may be desirable to have a simpler estimate of $T(n,d)$ that is free of binomial sums. For this purpose, we can simplify the conclusion of Theorem~\ref{thm: main} and state it as follows: 

\begin{corollary}			\label{cor: main}
  For any integers $n$ and $d$ such that $n > 1$ and $1 \le d \le n^{0.9}$, 
  the number of Boolean polynomial threshold functions $T(n,d)$ satisfies:
  $$
  \Big( 1 - \frac{C}{\log n} \Big)^d \cdot \frac{n^{d+1}}{d!}
  < \log_2 T(n,d) 
  < \frac{n^{d+1}}{d!}.
  $$ 
\end{corollary}

The upper bound in Corollary~\ref{cor: main} actually holds for all $n > 1$, $1 \le d \le n$.
We derive Corollary~\ref{cor: main} from Theorem~\ref{thm: main} in Section~\ref{s: corollary} 
by a careful analysis of the underlying binomial sums. 

For small degrees $d$, namely for $d = o(\log n)$, the factor $(1-C/\log n)^d$ becomes 
$1-o(1)$ and Corollary~\ref{cor: main} yields in this case the asymptotically tight bound on the capacity:
$$
\log_2 T(n,d) = \left( 1-o(1) \right) \cdot \frac{n^{d+1}}{d!}.
$$

\subsection{Prior work}			\label{s: prior work}
%--------------
Prior to our work, an upperbound on $T(n,d)$ that scales like $n^{d+1}/d!$ was known  \cite{Baldi,Anthony}. The  upperbounds given in Theorem~\ref{thm: main} and Corollary~\ref{cor: main} are more precise and general. 
The best known lower bound was given by M.~Saks \cite{Saks} in 1993:
\begin{equation}	\label{eq: Saks}
\log_2 T(n,d) \ge \binom{n}{d+1}.
\end{equation}
For all small degrees $d$, there is a multiplicative gap of size approximately $d+1$
between the optimal upper bound $n \binom{n}{\le d}$ in Theorem~\ref{thm: main}
and the lower bound \eqref{eq: Saks}.
For example, $T(n,2) \le n \binom{n}{\le 2} \approx n^3/2$ but 
\eqref{eq: Saks} only gives $T(n,2) \ge n^3/6$.
Our work closes this gap.

The asymptotically sharp result \eqref{eq: Zuev} about linear threshold functions 
has a remarkably short proof \cite{Zuev}.
It quickly follows from a combination of two results, one in enumerative combinatorics
and the other one in probability. The combinatorial result is a consequence of Zaslavsky's formula
for {\em hyperplane arrangements} \cite{Zaslavsky}, and the probabilistic result is Odlyzko's theorem
on spans of {\em random vectors} with independent $\pm 1$ coordinates \cite{Odlyzko}. 
Odlyzko's theorem is a stronger, {\em resilience} version of known results on the {\em singularity of random matrices}, results that state that a random matrix with $\pm 1$ entries has full rank with high probability. The original results on the singularity problem are due to J.~Koml\'os \cite{Komlos1, Komlos2}. More recently, the singularity problem has been actively studied in random matrix theory. A significant number of extensions and improvements on the result of J.~Koml\'os are now available, see e.g. \cite{KKS, Tao-Vu RSA, Costello-Tao-Vu, Tao-Vu JAMS, Rudelson Annals, Tao-Vu Annals, RV square, RV upper, Adamczak-et-al, RV rectangular, Tao-Vu GAFA, Bourgain-Wood-Vu, RV ICM, Nguyen EJP, Nguyen SIAM, V symmetric, RV JAMS, GNT, Basak-Rudelson, Tikhomirov Advances, Tikhomirov Israel, Litvak-et-al, Cook, Tikhomirov IMRN}, culminating in the very recent 
proof by K.~Tikhomirov 
\cite{Tikhomirov singularity}
providing optimal estimates for the probability 
that a random $\pm 1$ matrix be singular.

\subsection{Our approach}		
%--------------
Zuev's approach can be extended from linear threshold functions to polynomial threshold functions 
by lifting the problem into the tensor product space
$(\R^n)^{\otimes d}$. The combinatorial part of the argument generalizes
without any problem, but the probabilistic part is less obvious, because 
the theory of random tensors is not sufficiently developed yet. In \cite{baldi2018booleanv1} we developed some of theory in order to prove a version of Theorem \ref{thm: 
main}, but recently found a theorem by
E.~Abbe, A.~Shpilka, and A.~Wigderson \cite{ASW}
on the singularity of random tensors, developed in the context of Reed-Muller codes, that allows one to simplify the proof. The forthcoming paper \cite{V random tensors} gives an alternative, more general and quantitative, approach to the singularity problem for random tensors.
In the current paper, we prove a version of Odlyzko's resilience result \cite{Odlyzko} for 
random tensors.
Our proof of resilience is inspired by Odlyzko's proof and it uses the result of \cite{ASW}. 
The argument is non-trivial due to the lack of independence -- even if the entries of a random vector $x$ are stochastically independent, the entries of the random tensor $x^{\otimes d}$ are not.

\subsection{Outline of paper}
%------------
The rest of the paper is devoted to the proof of Theorem~\ref{thm: main}.
The next two sections deal with the combinatorial part of the argument. 
In Section~\ref{s: hyperplane arrangements}, we explain a canonical correspondence 
between linear threshold functions and hyperplane arrangements, and 
we discuss known bounds on the number of regions determined by hyperplane arrangements.
In Section~\ref{s: linearization}, we linearize polynomial threshold
functions by lifting them into the tensor product space, which then reduces
polynomial threshold functions to linear threshold functions and the corresponding hyperplane arrangements.
This allows us to quickly prove the (known) upper bound in Theorem~\ref{thm: main}
at the end of Section~\ref{s: linearization}.

Next we turn to the probabilistic part of the argument, in order to derive the lower bound. 
In Section~\ref{s: random tensors}, we explain 
the result of E.~Abbe, A.~Shpilka, and A.~Wigderson \cite{ASW} on the linear
independence of random tensors (Theorem~\ref{thm: ASW}), 
and state a new resilience version of this result (Theorem~\ref{thm: resilience}).
In Section~\ref{s: lower bound}, using this resilience version, we derive the lower bound of Theorem~\ref{thm: main}. In Sections~\ref{s: Littlewood-Offord} and  
\ref{s: resilience}, we prove the resilience result for random tensors. 
The argument uses the Littlewood-Offord Lemma for sums of independent 
random variables, so we explain the Littlewood-Offord lemma in 
Section~\ref{s: Littlewood-Offord}. 
In Section~\ref{s: resilience}, we prove the resilience result (Theorem~\ref{thm: resilience}), 
thus completing the entire argument.
In Section~\ref{s: corollary}, we deduce Corollary~\ref{cor: main}. 
In Section~\ref{s: further questions}, we describe several possible extensions and related open questions.

\subsection*{Acknowledgements}
%-------------
The authors are grateful to Michael Forbes who drew their attention to the 
work of E.~Abbe, A.~Shpilka, and A.~Wigderson \cite{ASW}, which lead 
to a great simplification of the original proof of Theorem~\ref{thm: main}. 
The authors also thank the anonymous reviewers for their useful comments and suggestions.

\section{Hyperplane arrangements}		\label{s: hyperplane arrangements}
%======================

There is a natural correspondence between threshold functions and regions of 
hyperplane arrangements, a classical topic in enumerative combinatorics that has been studied 
for decades \cite{Zaslavsky}; see \cite{Stanley}, \cite[Section 6]{Matousek}. 
To see the connection, let us fix a finite subset 
$S \subset \R^n \setminus \{0\}$ and consider all {\em homogeneous linear threshold functions} on $S$, 
i.e. functions $f : S \to \R$ of the form 
$$
f_a(x) = \sgn(\ip{a}{x})
$$
where $a \in \R^n$ is a fixed vector.
Consider the collection (``arrangement'') of hyperplanes 
$$
\{ x^\perp :\; x \in S \},
$$ 
where $x^\perp = \{z \in \R^n :\; \ip{z}{x}=0\}$ is the hyperplane through the origin with normal vector $x$. 
Two vectors $a$ and $b$ define the same homogeneous linear threshold function $f_a = f_b$ 
if and only if $a$ and $b$ 
lie on the same side of each of these hyperplanes. 
In other words, $f_a = f_b$ if and only if $a$ and $b$ lie in the 
same open component of the partition of $\R^n$ created by the hyperplanes $x^\perp$, with $x\in S$. 
Such open components are called the {\em regions} of the hyperplane arrangement (
Figure~\ref{fig: functions=regions}).
Thus we have the following lemma:

\begin{figure}[htp]			
  \centering 
  \includegraphics[width=.4\textwidth]{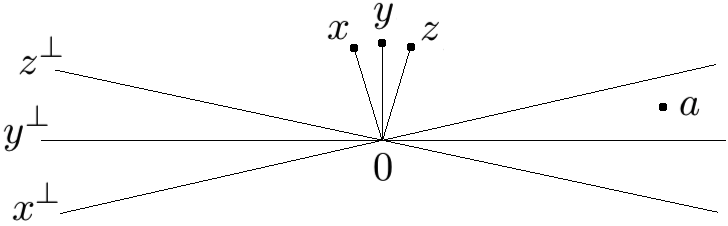} 
  \caption{The set of points $S = \{x,y,z\}$ in the plane defines the arrangement 
  	of hyperplanes (lines in this cases) $x^\perp, y^\perp, z^\perp$. 
	These hyperplanes partition the plane into six regions.
	Each region defines a different linear threshold function on $S$
	by the rule $f_a(x) = \sgn(\ip{a}{x})$, where $a$ is any point taken from that region.}
  \label{fig: functions=regions}
\end{figure}

\begin{lemma}[Threshold functions and hyperplane arrangements]	
		\label{lem: functions vs hyperplanes} 
  The number of homogeneous linear threshold functions on a given finite set 
  $S \subset \R^n \setminus \{0\}$
  equals the number of regions of the hyperplane arrangement $\{ x^\perp :\; x \in S \}$. 
\end{lemma}

This leads us to the following general question: how many regions are there in a given hyperplane arrangement?
An exact formula was found by Zaslavsky \cite{Zaslavsky}; see \cite{Stanley}. It expresses the number of regions via the M\"obius function of the poset of the intersection spaces associated with the hyperplanes. Computing the M\"obius function, however, may be a challenging task. 
Nevertheless, there are convenient bounds on the M\"obius function, which yields the following result:

\begin{lemma}[Counting regions of hyperplane arrangements]	\label{lem: Zaslavsky}
  Consider an arrangement of $p$ affine hyperplanes in $\R^m$, where $p \ge m$.
  Let $r(p,m)$ denote the number of regions of this arrangement.  
  \begin{enumerate}[ 1.]
    \item We have:
      \begin{equation}			\label{eq: general arrangements}
      r(p,m) \le \binom{p}{\le m}.
      \end{equation}
    \item \label{part: Zuev arrangements} 
    $r(p,m)$ is bounded below by the number of all intersection subspaces\footnote{An intersection subspace in this lemma refers to the intersection of any subfamily of the original hyperplanes. 
The dimensions of an intersection subspace may range from zero (a single point) to $m$ (intersecting an empty set of hyperplanes gives the entire space $\R^m$).
} defined by the hyperplanes. 
    \item If the hyperplanes are in general position, 
    then the upper and lower bounds are the same, and each bound becomes an equality. 

    \item \label{part: central}
    If all hyperplanes are central, i.e. pass through the same point, then
    the upper bound improves to
	\begin{equation}			\label{eq: central arrangements}
	r(p,m) \le 2 \binom{p-1}{\le m-1}.
	\end{equation}
     \item If the normal vectors to the hyperplanes are in general position, 
       then the inequality in \eqref{eq: central arrangements} becomes an equality. 
  \end{enumerate}
\end{lemma} 

To illustrate the first three parts of Lemma~\ref{lem: Zaslavsky}, 
consider first the line arrangement on the left in Figure~\ref{fig: arrangements}. 
This line arrangement is in general position,
it has seven regions, which is the same as $\binom{3}{\le 2}$ and also the same as the number of intersection subspaces, corresponding to: three points, three lines, and one plane. 
As for the last two parts of Lemma~\ref{lem: Zaslavsky}, consider the line arrangement on the right of Figure~\ref{fig: arrangements}. 
It has six regions but only five intersection subspaces. 
Moreover, since the normal vectors to the lines are in general position, 
part~\ref{part: central} of Lemma~\ref{lem: Zaslavsky} states that the number of regions must be 
bounded by $2 \binom{2}{\le 1} = 6$, which is sharp. 

\begin{figure}[htp]			
  \centering 
  \begin{subfigure}[b]{0.4\textwidth}
    \includegraphics[width=.6\textwidth]{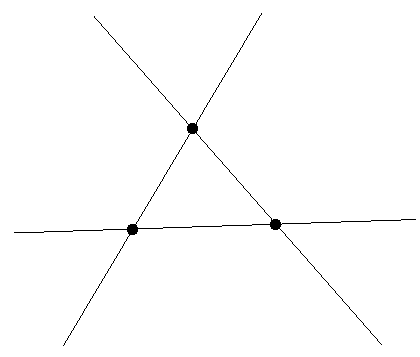} 
    %\caption{This covering of a pentagon $K$ by seven $\e$-balls shows that $\NN(K,\e) \le 7$.}
    %\label{fig: covering}
  \end{subfigure}
  \qquad
  \begin{subfigure}[b]{0.4\textwidth}
    \includegraphics[width=.6\textwidth]{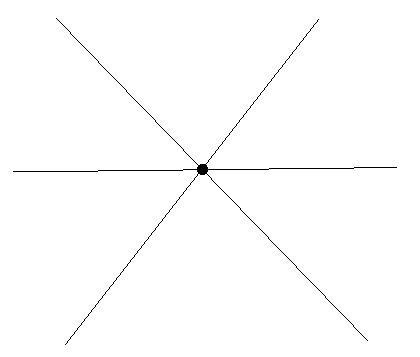} 
    %\caption{This packing of a pentagon $K$ by ten $\e$-balls shows that $\PP(K,\e) \ge 10$.}
  \end{subfigure}
  \caption{Two hyperplane arrangements in $\R^2$}
  \label{fig: arrangements}
\end{figure}

Lemma~\ref{lem: Zaslavsky} can be derived from Zaslavsky's formula \cite{Zaslavsky} and the basic properties of the M\"obius function of geometric lattices \cite[Section~7]{Rota}; see e.g. \cite[Proposition~2.4]{Stanley} for the derivation of \eqref{eq: general arrangements}
and \cite{Ojha} for \eqref{eq: central arrangements}. 
Originally, the upper bound \eqref{eq: general arrangements} goes back to 
R.~C.~Buck \cite{Buck} 
and the lower bound (property \ref{part: Zuev arrangements} in the lemma) was noted by Yu.~Zuev \cite{Zuev}.
Both the upper and lower bounds can also be proven directly -- without Zaslavsky's formula -- using simple inductive arguments, see \cite[Section 6]{Matousek} for the upper bound and \cite{Zuev} for the lower bound. 
The bound \eqref{eq: central arrangements} can also be proven by induction, as in the  
original proof of this result due to L.~Schl\"afli \cite[pp.~209--212]{Schlafli}, 
which is reproduced e.g. in \cite{Wendel} and, more explicitly, in \cite[Theorem~4.1]{Anthony}. See also
\cite{Winder66}.

\section{Tensor lift}		\label{s: linearization}
%---------------------------

Our next goal is 
to extend the correspondence between linear threshold functions and the regions associated with hyperplane arrangements to {\em polynomial} threshold functions of arbitrary degree. 
There is a very simple way to achieve this by lifting the problem into a tensor product space. 

\begin{definition}[$d$-th power of vectors and sets]
  The $d$-th power of a vector $x = (x_1,\ldots,x_n) \in \R^n$ 
  is the vector $x^{\le d} \in \R^{\binom{n}{\le d}}$
  whose coordinates are purely homogeneous monomials in $x_1,\ldots,x_n$.
  More precisely, the components of $x^{\le d}$ are indexed by 
  all subsets $I \subset [n]=\{1,2,\ldots,n\}$ that have at most $d$ elements, with:
  $$
  (x^{\le d})_I = \prod_{i \in I} x_i.
  $$
  The product over an empty set is defined to be $1$.  
  The $d$-th power of a set $S \subset \R^n$ is the set 
  $S^{\le d} \subset \R^{\binom{n}{\le d}}$ defined as 
  $$
  S^{\le d} = \left\{ x^{\le d} :\; x \in S \right\}.  
  $$
\end{definition}

For example, for $n=3$ and $d=2$ we have
$$
(x_1,x_2,x_3)^{\le 2} = (1, x_1, x_2, x_3, x_1 x_2, x_1 x_3, x_2 x_3).
$$
The notion of the $d$-th power of a vector is closely related to the notion of {\em tensor power} of a vector. 
Indeed, the $d$-th tensor power of $x \in \R^n$ can be identified with the vector 
$x^{\otimes d} \in \R^{n^d}$ indexed by {\em multisets} $I \subset [n]$ that have 
exactly $d$ elements, possibly with repetitions, and where the coefficients are defined similarly:
$$
(x^{\otimes d})_I = \prod_{i \in I} x_i.
$$
Although not as ubiquitous as tensor power, the concept of tensor lift 
$x^{\le d}$ appears naturally, although often implicitly, in coding theory. 
The list of all vectors $x^{\le d}$ for $x \in \{0,1\}^n$ be seen 
as the $\binom{n}{\le d} \times 2^n$ evaluation matrix, or a {\em truth table},
of all Boolean monomials of degree $d$ \cite{ASW}; the rows of this matrix can be identified 
with Walsh functions that arise in Fourier analysis of Boolean functions \cite{ODonnell book}.

The notion of $d$-th power allows us to establish a canonical correspondence 
between affine polynomials $p(x)$ in $n$ variables of degree at most $d$ with homogeneous monomials 
and homogeneous linear functions of $x^{\le d}$:
$$
p(x) = \sum_{I: |I| \le d} a_I \prod_{i \in I} x_i = \ip{a}{x^{\le d}}, 
\quad \text{where} \quad
x \in \R^n, \; 
a \in \R^{\binom{n}{\le d}}.
$$
This correspondence give the following Lemma.

\begin{lemma}[Linearizing polynomial threshold functions]			\label{lem: linearizing}
  Let $S \subset \{-1,1\}^n$. 
  There are as many different polynomial threshold functions of degree $d$ on $S$ 
  as there are homogeneous linear threshold functions on $S^{\le d}$.
\end{lemma}

This simple fact and the link to hyperplane arrangements described 
in Section~\ref{s: hyperplane arrangements} leads to the following
upper known bound on $T(n,d)$, which is implicit in \cite{Baldi}; see \cite[Theorem~4.7]{Anthony}.

\begin{theorem}[Upper bound]			\label{thm: upper bound sharp}
  For any $1 \le d \le n$, we have:
  $$
  T(n,d) \le 2 \binom{2^n-1}{\le m-1}
  \quad \text{where} \quad
  m = \binom{n}{\le d}.
  $$
\end{theorem}

\begin{proof}
By Lemma~\ref{lem: linearizing}, the number of Boolean polynomial threshold functions
$T(n,d)$ equals the number of 
homogeneous threshold functions on $\left( \{-1,1\}^n \right)^{\le d}$.
By Lemma~\ref{lem: functions vs hyperplanes}, this is the same as the 
number of regions of the arrangement of hyperplanes 
$(x^{\le d})^\perp$, $x \in \{-1,1\}^n$.
These are $2^n$ central hyperplanes in $\R^m$, where $m = \binom{n}{\le d}$.
Thus part~\ref{part: central} of Lemma~\ref{lem: Zaslavsky} yields:
$$
T(n,d) \le r(2^n,m-1) \le 2 \binom{2^n-1}{\le m-1},
$$
as claimed.
\end{proof}

By simplifying this bound, we can prove the upper bound of Theorem~\ref{thm: main}:

\begin{proof}[Proof of the upper bound in Theorem~\ref{thm: main}]
Due to Theorem~\ref{thm: upper bound sharp}, 
it is enough to check that 
$$
B(n,m) \coloneqq 2 \binom{2^n-1}{\le m-1} \le 2^{nm}
$$
for all integers $n \ge 1$ and $m \ge 2$. 

If $m \ge 4$, then an elementary bound on the binomial coefficients (Lemma~\ref{lem: binomial bounds} in the Appendix) gives
$$
B(n,m) 
\le 2 \Big( \frac{e2^n}{m-1} \Big)^{m-1}
\le 2 \cdot 2^{n(m-1)}
\le 2^{nm}.
$$
Here we used that $m-1 \ge 3 \ge e$.

The two remaining cases to check are $m=2$ and $m=3$. 
If $m=2$, then, setting $N \coloneqq 2^n$, we get
$$
B(n,m) = 2 \binom{N-1}{\le 1} = 2N \le N^2 = 2^{nm},
$$
since $N = 2^n \ge 2$. 
Finally, if $m=3$, then we similarly have
$$
B(n,m) = 2 \binom{N-1}{\le 2} = N^2-N+2 
\le N^2 \le N^3 = 2^{nm}.
$$
Theorem~\ref{thm: main} is proved.
\end{proof}

\section{Random tensors}				\label{s: random tensors}
%================

\subsection{General position}
%--------------
It remains to prove the lower bound in Theorem~\ref{thm: main}. 
Let us try to reverse the argument that we gave for the upper bound in 
the proof of Theorem~\ref{thm: upper bound sharp}. There we noted that
$T(n,d)$ is the same as the number of regions of the hyperplane arrangement 
$x^\perp$, $x \in \XX$, where
$$%\begin{equation}	\label{eq: XX}
\XX \coloneqq \left( \{-1,1\}^n \right)^{\le d} \subset \R^{\binom{n}{\le d}}.
$$%\end{equation}
Suppose for a moment that the points in $\XX$ are in general position (in reality they are not).
Then every subset of $\XX$ consisting of $m < \binom{n}{\le d}$ points
would span a different subspace in $\R^{\binom{n}{\le d}}$. 
The orthogonal complements of these different subspaces are different, too. 
These complements are intersections of some hyperplanes from our arrangement  
$x^\perp$, $x \in \XX$;
we called them {\em intersection subspaces} in Lemma~\ref{lem: Zaslavsky}.
According to part~\ref{part: Zuev arrangements} of this lemma, 
the number of regions of our hyperplane arrangement is bounded 
below by the number of all intersection subspaces, which 
is at least as many as there are $m$-element subsets of $\XX$. This gives:
$$
T(n,d) \ge \binom{|\XX|}{m} = \binom{2^n}{\binom{n}{\le d}-1}
$$
if we choose $m = \binom{n}{\le d}-1$.
This would be an almost matching lower lower bound for the upper bound  
in Theorem~\ref{thm: upper bound sharp}.

\subsection{Linear independence}
%----------

The problem with this heuristic argument is that the points in the $d$-th power of the Boolean hypercube $\left( \{-1,1\}^n \right)^{\le d}$, and even in the Boolean hypercube $\{-1,1\}^n$ itself, 
are very far from being in general position. 
For example, the affine hyperplane spanned by a $(n-1)$-dimensional face of the Boolean cube 
contains $2^{n-1}$ points.
Nevertheless, we might be able to say that most subsets of points are in the general position. 
This is where probabilistic reasoning becomes useful, allowing us to interpret ``most'' as {\em random}.

The first result in this direction was recently proved by 
E.~Abbe, A.~Shpilka, and A.~Wigderson \cite[Theorem~4.5]{ASW}
in the context of their study of Reed-Solomon codes. 
Their result states that a random subset $\left( \{-1,1\}^n \right)^{\le d}$
is linearly independent with high probability:

\begin{theorem}[Linear independence \cite{ASW}]		\label{thm: ASW}
  Let $n,d,m$ be positive integers such that 
  $$
  m < \binom{n - \log\binom{n}{\le d} - t}{\le d}.
  $$
  Consider independent random vectors $x_1,\ldots,x_m$
  uniformly distributed in $\{-1,1\}^n$. 
  Then, with probability larger than $1-2^{-t}$, the random vectors 
  $x_1^{\le d}, \ldots, x_m^{\le d} \in \R^{\binom{n}{\le d}}$ are linearly independent.
\end{theorem}

The special case of Theorem~\ref{thm: ASW} obtained with $d=1$ states 
that a $n \times m$ random matrix with columns $x_k$ has full rank with high probability, 
if $m < n - C \log n$. 
This statement is not difficult to prove, and much more precise results are known 
about the singularity of random matrices 
\cite{KKS, Tao-Vu RSA, Costello-Tao-Vu, Tao-Vu JAMS, Rudelson Annals, Tao-Vu Annals, RV square, RV upper, Adamczak-et-al, RV rectangular, Tao-Vu GAFA, Bourgain-Wood-Vu, RV ICM, Nguyen EJP, Nguyen SIAM, V symmetric, RV JAMS, GNT, Basak-Rudelson, Tikhomirov Advances, Tikhomirov Israel, Litvak-et-al, Cook, Tikhomirov IMRN}. 

The original version of Theorem~\ref{thm: ASW} from \cite{ASW} guaranteed 
linear independence over the finite field $\F_2$, which is stronger
than the linear independence over $\R$.\footnote{If some Boolean vectors are linearly dependent over $\R$, then one can find a non-trivial linear combination that equals zero, and whose all coefficients are rational, and thus even integer. (This can be seen e.g. by performing the Gauss elimination.) 
Moreover, without loss of generality, not all of the coefficients are even: otherwise we can divide both sides by $2$. Taking $\mod 2$ of both sides of this equation, we obtain a linear dependence over $\F^2$.}
Moreover, in the original version of this theorem, the coordinates of the vectors 
$x_1,\ldots,x_m$ were uniformly distributed in $\{0,1\}$ rather than in $\{-1,1\}$,
but the proof can be easily adapted to random $\{-1,1\}$ valued random variables.\footnote{Most of the argument of \cite{ASW} 
	extends to $\{-1,1\}$ without any change. The only place that needs attention 
	is Lemma~4.10 \cite{ASW}, which states that there exists a lot of linearly independent 
	polynomials on any large subset of $\F_2^m$. Obviously, there is the same number of 
	polynomials on any affine transformation of $\F_2^m$, in particular on $\{-1,1\}^m$.}

\subsection{Resilience}
%----------

However, linear independence of random vectors $x_1^{\le d}, \ldots, x_m^{\le d}$ 
is still too weak for our purposes. 
What we really need is
%% for such vectors to span many different subspaces.
%%This would clearly be true if we can 
to be able to show that the span of random vectors 
$x_1^{\le d}, \ldots, x_m^{\le d}$ does not contain any vector of the form $u^{\le d}$, $u \in \{-1,1\}^n$, that is different from all the $\pm x_i^d$. 
This is a {\em resilience} property of linear independence, as it states that not only 
the set of vectors are independent, but independence holds even if we add any vector 
$u^{\le d}$, $u \in \{-1,1\}^n$, to the set. 
The following new result establishes the resilience of linear independence for random tensors:

\begin{theorem}[Resilience of linear independence]	\label{thm: resilience}
  Let $n,d,m$ be positive integers such that 
  $d n^{0.02} \le t \le 0.001n$ 
  and 
  $$
  m < \binom{n - \frac{Cn}{\log n} - t}{\le d}.
  $$
  Consider independent random vectors $x_1,\ldots,x_m$
  uniformly distributed in $\{-1,1\}^n$. 
  Then, with probability larger than $1-2^{-t/4}$, the  span of the random vectors 
  $x_1^{\le d}, \ldots, x_m^{\le d}$ does not contain any other vector of the form $u^{\le d}$,
  $u \in \{-1,1\}^n$, that is different from $\pm x_i^{\le d}$. 
\end{theorem}

This result is not an easy consequence of Theorem~\ref{thm: ASW}, 
and it requires additional probabilistic tools. 
The partial case of Theorem~\ref{thm: resilience} where $d=1$
was proved by Odlyzko \cite{Odlyzko} with a very sharp bound on the probability, 
which we do not need here.
It was later noticed that for $d=1$, the probabilities for linear independence and resilience 
are asymptotically equivalent \cite{Voigt-Ziegler}.
Odlyzko's result was used in Zuev's argument \cite{Zuev} to prove \eqref{eq: Zuev}.
For tensors of any order $d > 1$, Theorem~\ref{thm: resilience} is new. We derive it
from Theorem~\ref{thm: ASW} using an argument that is inspired by Odlyzko's method.

\subsection{Proof of the lower bound in Theorem~\ref{thm: main}.}	\label{s: lower bound}
%-------------------------------

Let us assume for a moment that the Resilience Theorem~\ref{thm: resilience} is valid,
and show how it yields the lower bound in our main result, Theorem~\ref{thm: main}.

\begin{lemma}[Lots of unique subspaces]			\label{lem: lots of uniques}
  Let $n,d,m$ be positive integers such that 
  $1 \le d \le n^{0.9}$ 
  and 
  $$
  m \le \binom{n - \frac{Cn}{\log n}}{\le d}.
  $$
  Then there exist at least $\frac{1}{2}\binom{2^n}{m}$ 
  different subspaces of the form $\Span(S^{\le d})$ 
  where $S$ are subsets of $\{-1,1\}^n$ of cardinality $m$.
\end{lemma}

\begin{proof}
Given a subset $S \subset \{-1,1\}^n$, let us call it {\em good} 
if $\Span(S^{\le d})$ does not contain any vector of the form $u^{\le d}$,
$u \in \{-1,1\}^n$, that is different from all $\pm x^{\le d}$, $x \in S$. 
Call $S$ {\em bad} otherwise. 
Obviously, if $S$ and $T$ are two different good subsets, 
then $\Span(S^{\le d})$ and $\Span(T^{\le d})$ are two different  subspaces. 
Thus, to complete the proof, it suffices to show that at least half of all $m$-element subsets $S$ of the Boolean hypercube $\{-1,1\}^n$ are good.

Let us apply Theorem~\ref{thm: resilience} with the 
value $t \coloneqq Cn/\log n$, which obviously satisfies the assumptions for sufficiently
large $n$. The theorem implies that 
if $S$ is a random set obtained by sampling $m$ points from $\{-1,1\}^n$ {\em with replacement}, then: 
\begin{equation}	\label{eq: bad}
\Pr{S \text{ is bad}} \le \frac{1}{4}.
\end{equation}
The probability that there are no repetitions among these $m$ points is:
\begin{equation}	\label{eq: no repeat}
\Pr{ \text{no repeat} } 
  = \prod_{i=1}^{m-1} \Big( 1 - \frac{i}{2^n} \Big)
  \ge \Big( 1 - \frac{m-1}{2^n} \Big)^{m-1}
  \ge 1 - \frac{(m-1)^2}{2^n} \ge \frac{1}{2}.
\end{equation}
To check the last bound, recall that $d \le n^{0.9}$, so for sufficiently large
$n$ we get:
$$
\log m \le \log \binom{n}{\le d} \le d \log (en) \le \frac{n-1}{2},
$$
which implies $(m-1)^2 \le 2^{n-1}$ and yields the bound in \eqref{eq: no repeat}.

Thus, the probability that a random $m$-element subset $S$, 
sampled {\em without replacement} from $\{-1,1\}^n$, is bad, is be bounded by:
$$
\frac{\Pr{S \text{ is bad}}}{\Pr{ \text{no repeat} }} \le \frac{1/4}{1/2} = \frac{1}{2}.
$$
Hence, at most a half of the $m$-elements subsets of $\{-1,1\}^n$ are bad. 
The proof of the lemma is complete.
\end{proof}

\begin{proof}[Proof of the lower bound in Theorem~\ref{thm: main}.]
According to the Linearization Lemma~\ref{lem: linearizing}, 
$T(n,d)$ is the number of linear threshold functions on the set
$$
\XX := \left( \{-1,1\}^n \right)^{\le d}
\subset \R^{\binom{n}{\le d}}.
$$
This number, according to Lemmas~\ref{lem: functions vs hyperplanes} and \ref{lem: Zaslavsky},
is bounded below by the number of all intersection subspaces, which are the linear subspaces generated by intersecting various hyperplanes $z^\perp$, $z \in \XX$.
The orthogonal complement of each intersection subspace is the linear span of a subset of $\XX$. Thus, the number of intersection subspaces equals the number of subspaces obtained as spans of subsets of $\XX$. The number of spans can be bounded below using 
Lemma~\ref{lem: lots of uniques}. This line of reasoning yields:
$$
T(n,d) \ge \frac{1}{2}\binom{2^n}{m}
\quad \text{where} \quad
m = \binom{n - \frac{Cn}{\log n}}{\le d}.
$$
It remains to simplify this bound. Taking logarithms of both sides and using 
a simple bound on binomial coefficients (Lemma~\ref{lem: binomial bounds} in the Appendix),
we get:
$$
\log_2 T(n,d) \ge m(n-\log m)-1
\ge m(n-2\log m).
$$ 
Another elementary bound on binomial sums (Lemma~\ref{lem: more binomial bounds} in the Appendix) gives
$$
m \ge \Big( 1 - \frac{2C}{\log n} \Big)^d \binom{n}{\le d},
$$
and using Lemma~\ref{lem: binomial bounds} again we see that
$$
\log m < \log \binom{n}{\le d} \le d \log (en) \le \frac{Cn}{\log n}.
$$
Thus
$$
\log_2 T(n,d) 
> \Big( 1 - \frac{2C}{\log n} \Big)^{d+1} n \binom{n}{\le d}
\ge \Big( 1 - \frac{4C}{\log n} \Big)^d n \binom{n}{\le d}.
$$
This competes the proof of the main theorem.
\end{proof}

\section{The Littlewood-Offord lemma}			\label{s: Littlewood-Offord}
%===========

In this section and the next one, we prove the Resilience Theorem~\ref{thm: resilience}.
Our argument is inspired by the proof of the partial case of this result for $d=1$
due to Odlyzko \cite{Odlyzko}. The proof is based on the classical Littlewood-Offord lemma
about {\em anti-concentration} of sums of independent random variables. 

Let $\xi_1,\ldots,\xi_n$ be independent random variables and $a_1,\ldots,a_n \in \R$ be fixed coefficients. 
A classical question, which goes back to J.~E.~Littlewood and A.~C.~Offord \cite{Littlewood-Offord} is 
to determine the probability that the sum of independent random variables $\sum a_k \xi_k$ hits a given number $u \in \R$. The first general result on this problem, now commonly known as the Littlewood-Offord Lemma, was proved by J.~E.~Littlewood and A.~C.~Offord \cite{Littlewood-Offord} and sharpened by P.~Erd\"os \cite{Erdos}. 

\begin{lemma}[Littlewood-Offord Lemma \cite{Erdos}] 		\label{lem: Littlewood-Offord}
  Let $\xi_1, \ldots, \xi_n$ be independent, zero-mean, random variables 
  taking values in $\{-1, 1\}$,
  and let $a_1,\ldots,a_n$ be nonzero real numbers.
  Then, for every fixed $u \in \R$, we have\footnote{Here $\lfloor m \rfloor$ denotes the 
    floor of $m$, i.e. the largest integer that is less or equal to $m$.}
  $$
  \Pr{\sum_{k=1}^n a_k \xi_k =u} \le 2^{-n} \binom{n}{\lfloor n/2 \rfloor} =: P(n).
  $$
\end{lemma}

A slightly more general version of Lemma~\ref{lem: Littlewood-Offord}, which bounds the probability that 
the sum falls in a given neighborhood of $u$, quickly follows from Sperner's theorem in combinatorics \cite{Erdos}, see \cite[Chapter~4]{Bollobas}.

Note that the probability bound in the Littlewood-Offord lemma is sharp: it reduces to an equality 
if all coefficients $a_k$ are the same and $u = 0$.
For many other vectors of coefficient $a = (a_1,\ldots,a_n)$, one can obtain better bounds depending on the arithmetic structure of $a$. Such bounds have been extensively studied in connection to number theory, combinatorics and, more recently, random matrix theory; see, for instance,  \cite{Erdos, Sarkozy-Szemeredi, Halasz, Frankl-Furedi, Tao-Vu Annals, RV square, RV rectangular, Tao-Vu sharp, Nguyen-Vu, V symmetric, Nguyen Duke, Costello, MNV, RV no-gaps}, and the surveys \cite{Tao-Vu Bulletin, RV ICM}.

Using Stirling's approximation to estimate the binomial coefficient, we can derive the following, less precise but simpler, bound on the probability in the Littlewood-Offord Lemma.

\begin{lemma}[Probability bounds in Littlewood-Offord Lemma]		\label{lem: bounds on Pn}
  We have: 
  $$
  P(n) \le \frac{C}{\sqrt{n}} \text{ for all } n \ge 1; 
  \qquad 
  P(n) \le \frac{3}{8} \text{ for all } n \ge 3.
  $$
\end{lemma}

\begin{proof}
The first bound follows from Stirling's formula. Furthermore, one can easily check 
that the numbers $P(n)$ form a non-increasing sequence and $P(3) = 3/8$. 
This gives the second bound.
\end{proof}

The bound in the Littlewood-Offord Lemma~\ref{lem: Littlewood-Offord} can be slightly strengthened
if $u \ne 0$. Although the following may seem like a small improvement, it can be critical for small 
values of $n$.

\begin{lemma}		\label{lem: Littlewood-Offord nonzero RHS}
  If $u \ne 0$ in the Littlewood-Offord Lemma~\ref{lem: Littlewood-Offord}, then 
  $$
  \Pr{\sum_{k=1}^n a_k \xi_k =u} \le P(n+1).
  $$
\end{lemma}

\begin{proof}
Let $\xi_{n+1}$ be a mean zero random variable taking values in $\{-1,1\}$, 
and which is independent of $\xi_1,\ldots,\xi_n$. Then
\begin{align*} 
\Pr{\sum_{k=1}^n a_k \xi_k = u} 
&= \Pr{\sum_{k=1}^n a_k \xi_k = u \xi_{n+1}} \quad \text{(by symmetry)} \\
&= \Pr{\sum_{k=1}^{n+1} a_k \xi_k = 0}		\quad \text{(where $a_{n+1} := -u$)} \\
&\le P(n+1). 
\end{align*} 
The proof is complete.
\end{proof}

\section{Resilience}  \label{s: resilience}
%-----------------

To prove Theorem~\ref{thm: resilience}, we have to show that it is unlikely
that there exists a vector $u \in \{-1,1\}^n$,
and coefficients $a_1,\ldots,a_m \in \R$ at least two of which are non-zero, such that:
\begin{equation}	\label{eq: combination}
%\sum_{k=1}^m a_k = 1 
%\quad \text{and} \quad
\sum_{k=1}^m a_k x_k^{\le d} = u^{\le d}.
\end{equation}
Our argument will be a little different depending on how many coefficients $a_k$ are nonzero. We will first analyze the case of ``long combinations'' where at least $n^{0.01}$ coefficients are nonzero, and then the remaining case of ``short combinations''. 

\subsection{Long combinations}			\label{s: long}
%...................

Let $\Plong$ denote the probability 
that there exists a vector $u \in \{-1,1\}^n$ and coefficients $a_1,\ldots,a_m \in \R$ 
at least $n^{0.01}$ of which are nonzero, and such that \eqref{eq: combination} holds. 
Our goal is to bound $\Plong$.

\medskip

{\em Step 1. Extracting two batches of equations.}
We can view \eqref{eq: combination} 
as a system of $\binom{n}{\le d}$ linear equations 
in variables $a_1,\ldots,a_m$, and we can write it as
$$
\sum_{k=1}^m a_k \big( x_k^{\le d} \big)_I = \big( u^{\le d} \big)_I, 
\quad \text{or equivalently as} \quad
\sum_{k=1}^m a_k \prod_{i \in I} x_{ki} = \prod_{i \in I} u_{i}, 
$$
for each subset $I \subset [n]$ with at most $d$ elements. 
We will consider two subsets, or ``batches'', of these equations. 
The first batch will be used to determine the coefficients $(a_k)$, 
and the second batch will be used to bound the probability. 

Fix an integer $1 \le n_0 \le n$ whose value will be determined later. 
The first batch will be defined by the subsets:
%%\marginpar{Define}
$$
I \subset \binom{[n_0]}{\le d}
$$ 
and the second batch, by the subsets: 
$$
I \subset \binom{n_0+1,\ldots,n}{1}.
$$
Thus, $\sum_{k=1}^m a_k \prod_{i \in I} x_{ki}$ is a polynomial in the first $n_0$
variables $x_1,\ldots,x_{n_0}$ in the first batch, and we have a linear form in the remaining variables $x_{n_0+1}, \ldots, x_n$ in the second batch.

This gives us two systems of stochastically independent equations. 
We can rewrite them as follows. The first batch is given by:
\begin{equation}	\label{eq: first batch rewritten}
\sum_{k=1}^m a_k \bar{x}_k^{\le d} = \bar{u}^{\le d}
\end{equation}
where the vector $\bar{u} \in \{-1,1\}^{n_0}$ is obtained from $u \in \{-1,1\}^n$ 
by keeping only the first $n_0$ coefficients, and similarly for the vectors 
$\bar{x}_k \in \{-1,1\}^{n_0}$. 
The second batch is given by:
$$
\sum_{k=1}^m a_k x_{ki} = u_i, \; n_0 < i \le n.
$$

%\begin{equation}	\label{eq: second batch simple}
%\sum_{k=1}^m a_k x_{ki} = u_i,
%\quad n_0 < i < n.
%\end{equation}

\medskip

{\em Step 2: The first batch determines the coefficients, the second batch bounds the probability.}
Suppose that $n_0$ is chosen so that:
\begin{equation}	\label{eq: m desired long}
m < \binom{n_0 - \log\binom{n_0}{\le d} - t}{\le d}.
\end{equation}
Then, by Theorem~\ref{thm: ASW}, the random vectors 
$\bar{x}_1^{\le d},\ldots,\bar{x}_m^{\le d}$
are linearly independent with probability larger than $1-2^{-t}$. 
In this case, the first batch of equations has full rank. 
Let us condition on a realization of $\bar{x}_1, \ldots, \bar{x}_m$
for which the linear independence does hold.

Suppose the event $\Plong$ occurs. 
The vector $u \in \{-1,1\}^n$ in \eqref{eq: combination} 
can be chosen in $2^n$ ways; let us fix it. 
Since $\bar{x}_k^{\le d}$ and $\bar{u}^{\le d}$ are fixed at this point, 
linear independence implies that the coefficients $(a_k)$ are uniquely determined by 
the first batch of equations.
Thus the coefficients $(a_k)$ are now fixed, too. 
Put them in the second batch of equations, 
which is stochastically independent from the first. 
This reasoning gives
\begin{equation}	\label{eq: Plong prelim}
\Plong \le 2^{-t}
  + 2^n \cdot \max_{a,u} \Pr{ \sum_{k=1}^m a_k x_{ki} = u_i, \; n_0 < i \le n }
\end{equation}
where the maximum is over all vectors $a = (a_k)$ with at least $n^{0.01}$ nonzero coefficients, 
%that satisfy $\sum_{k=1}^m a_k = 1$, 
and over all vectors $u = (u_i)$ with $\pm 1$ coefficients.

\medskip

{\em Step 3: Applying the Littlewood-Offord Lemma.}
The Littlewood-Offord Lemma (Lemma~\ref{lem: Littlewood-Offord} and Lemma~\ref{lem: bounds on Pn}) can help us bound the probability of each equation in \eqref{eq: Plong prelim}. 
Indeed, since at least $n^{0.01}$ coefficients $a_k$ are nonzero, we get 
$$
\Pr{ \sum_{k=1}^m a_k x_{ki} = u_i } \le P(n^{0.01}) \le \frac{C'}{n^{1/8}}
$$
for each $i$. Since all $n-n_0$ such equations in \eqref{eq: Plong prelim} are stochastically independent, this implies
$$
\Plong 
\le 2^{-t} + 2^n \cdot \Big( \frac{C'}{n^{1/8}} \Big)^{n-n_0}.
$$

Now is a good time to select a value for $n_0$. Let us set:
\begin{equation}	\label{eq: n0 long}
n_0 = n - \frac{C''n}{\log n}
\end{equation}
where the absolute constant $C''$ is large enough. 
This choice allows us to have:
$2^n \cdot ( C''/n^{1/8} )^{n-n_0} \le 2^n \cdot 2^{-2n} = 2^{-n}$.
Furthermore, since $t \le n$ by assumption, $2^{-n} \le 2^{-t}$, and we obtain the bound:
$$
\Plong 
\le 2^{1-t}.
$$

\medskip

{\em Step 4: Checking the bound on $m$.}
It remains to check that our choice of $n_0$ also 
makes our bound \eqref{eq: m desired long} valid.
An elementary bound on binomial sums (Lemma~\ref{lem: binomial bounds} in the Appendix)
yields:
$$
\log \binom{n}{\le d} 
\le d \log (en_0) 
\le d \log (en) 
\le \frac{n}{\log n}
$$
by the assumption on $d$. Here we also used that $n$ can be assumed to be sufficiently large
(larger than any given absolute constant); otherwise the assumption on $m$ in Theorem~\ref{thm: resilience} is vacuous if the absolute constant $C$ is chosen large enough. Similarly,
$m$ and $t$ can be assumed to be sufficiently large. 
We will repeatedly use this in the rest of the argument. 

This and our choice \eqref{eq: n0 long} of $n_0$ give:
$$
\binom{n_0 - \log\binom{n_0}{\le d} - t}{\le d}
\ge \binom{n - \frac{(C''+1) n}{\log n} - t}{\le d}
> m
$$
where the last inequality is the theorem's assumption on $m$ 
with $C = C''+1$.
We have checked that \eqref{eq: m desired long} is valid, 
and thus completed the analysis of the long combinations.

\subsection{Short combinations}
%...................

Let $\Pshort$ denote the probability 
that there exists a vector $u \in \{-1,1\}^n$, and coefficients $a_1,\ldots,a_m \in \R$ 
at least two and at most $n^{0.01}$ of which are nonzero, such that \eqref{eq: combination} holds. Our goal is to bound $\Pshort$.

\medskip

{\em Step 1: Fixing the pattern of non-zero coefficients.}
Let us first address a simpler problem where the pattern of non-zero coefficients $a_k$ is fixed. 
Namely, let us require that the non-zero coefficients $a_k$ be exactly the first $m_0$ ones, where 
$$
m_0 \in [2,n^{0.01}]
$$ 
is a fixed integer. Denote the probability in this simplified problem by $P_{m_0}$.
Thus, $P_{m_0}$ is the probability that there exists a vector $u \in \{-1,1\}^n$ and nonzero coefficients $a_1,\ldots,a_{m_0} \in \R$, satisfying:
\begin{equation}	\label{eq: combination short}
%\sum_{k=1}^{m_0} a_k = 1 
%\quad \text{and} \quad
\sum_{k=1}^{m_0} a_k x_k^{\le d} = u^{\le d}.
\end{equation}

\medskip

{\em Step 2: Two batches of equations.}
We define and analyze two batches of equations in the same way as in our analysis 
of long combinations in Section~\ref{s: long}, except that the value of $n_0$ in this case will be chosen differently later in the proof. Indeed, if:
\begin{equation}	\label{eq: m desired short}
m_0 < \binom{n_0 - \log\binom{n_0}{\le d} - t}{\le d},
\end{equation}
repeating the argument from Section~\ref{s: long} we get:
\begin{equation}	\label{eq: Pshort prelim}
P_{m_0} \le 2^{-t}
  + 2^n \cdot \max_{a,u} \Pr{ \sum_{k=1}^m a_k x_{ki} = u_i, \; n_0 < i \le n }
\end{equation}
where the maximum is over all vectors $a = (a_k)$ with all nonzero coefficients, 
%that satisfy $\sum_{k=1}^m a_k = 1$, 
and over all vectors $u = (u_i)$ with $\pm 1$ coefficients.

\medskip

{\em Step 3: Applying the Littlewood-Offord Lemma.}
The Littlewood-Offord Lemma (Lemma~\ref{lem: Littlewood-Offord nonzero RHS} and Lemma~\ref{lem: bounds on Pn}) can help us bound each probability term in \eqref{eq: Plong prelim}. 
Indeed, since all $a_k$ and $u_i$ are nonzero,
we get for each $i$:
$$
\Pr{ \sum_{k=1}^{m_0} a_k x_{ki} = u_i } 
\le P(m_0+1) \le \frac{3}{8}.
$$
Since all $n - n_0$ such equations in \eqref{eq: Pshort prelim} are stochastically independent, this implies:
$$
P_{m_0} \le 2^{-t}
  + 2^n \Big( \frac{3}{8} \Big)^{n - n_0}.
$$

Now is a good time to select a value for $n_0$. Let us set:
\begin{equation}	\label{eq: n0 short}
n_0 = 0.1 n.
\end{equation}
This choice guarantees that $2^n (3/8)^{n-n_0} \le (0.82)^n$. 
Furthermore, since $t \le n/4$ by assumption, we have $(0.82)^n \le 2^{-t}$. 
Thus we conclude that: 
$$
P_{m_0} \le 2^{1-t}.
$$

\medskip

{\em Step 4: Unfixing the pattern of non-zero coefficients.}
In the beginning of the proof, we made a simplifying assumption that 
the support of the coefficient vector $a = (a_k)$ be the set $[m_0]$. 
The same argument holds if we replace $[m_0]$ by any other subset of $[m]$ of cardinality $m_0$. 
Thus, taking the union bound over all $\binom{m}{m_0}$ ways of choosing the support of $a$,
and over all $m_0 \in [2,n^{0.01}]$ allowed sizes of the support, we obtain:
$$
\Pshort \le \sum_{m_0=2}^{n^{0.01}} \binom{m}{m_0} P_{m_0} 
\le \binom{m}{n^{0.01}} 2^{1-t}.
$$
Recall that $n$, $t$ and $m < \log\binom{n}{\le d}$ are sufficiently large, 
so using elementary bounds on the binomial coefficients 
(Lemma~\ref{lem: binomial bounds} in the Appendix) we get:
$$
\log\binom{m}{n^{0.01}}
\le 2 n^{0.01} d \log n
\le t/2
$$
by the assumption on $t$. This allows us to conclude that:
$$
\Pshort \le 2^{t/2} 2^{1-t} = 2^{1-t/2} \le 2^{-t/3}.
$$

\medskip

{\em Step 4: Checking the bound on $m_0$.}
It remains to be checked that our choice of $n_0$ is consistent with the bound \eqref{eq: m desired short}.
Note that using elementary bounds on binomial coefficients (Lemma~\ref{lem: binomial bounds}):
$$
\log \binom{n_0}{\le d} 
\le d \log (en_0) 
\le d \log n 
\le \frac{0.01n}{3} = \frac{n_0}{3}
$$
by our assumption on $d$. Similarly: 
$$
t \le 0.001n \le \frac{0.01n}{3} = \frac{n_0}{3}.
$$
This yields:
$$
\binom{n_0 - \log\binom{n_0}{\le d} - t}{\le d}
> \binom{n_0/3}{\le d}
\ge \frac{n_0}{3}
\ge 0.01 n
\ge m_0
$$
where the last inequality follows for large $n$ from the constraint 
$m_0 \le n^{0.01}$.
Thus we have checked that \eqref{eq: m desired long} is valid, 
completing the analysis of the case of short combinations.

\medskip

{\em Step 5: Conclusion of the proof of Theorem~\ref{thm: resilience}.}
Combining our results for long and short combinations, we obtain that
the overall failure probability is at most
$$
\Plong + \Pshort \le 2^{1-t} + 2^{-t/3}
\le 2^{-t/4}
$$
for sufficiently large $t$.
This concludes the proof of Theorem~\ref{thm: resilience}.
\qed

\section{Proof of Corollary~\ref{cor: main}} 		\label{s: corollary}
%=========================

We will derive both lower and upper bounds from Theorem~\ref{thm: main}. 
Let us start from the lower bound, which is much simpler. 

\subsection{Lower bound}
%--------------

Note that
$$
\binom{n}{\le d}
> \binom{n}{d}
= \frac{n(n-1)\cdots(n-d+1)}{d!}
> \frac{(n-d)^d}{d!}
= \Big( 1 - \frac{d}{n} \Big)^d \cdot \frac{n^d}{d!}.
$$
Then Theorem~\ref{thm: main} yields
$$
\log_2 T(n,d)
> \Big[ \Big( 1 - \frac{C}{\log n} \Big) \Big( 1 - \frac{d}{n} \Big) \Big]^d \cdot \frac{n^{d+1}}{d!}
\ge \Big( 1 - \frac{2C}{\log n} \Big)^d \cdot \frac{n^{d+1}}{d!},
$$
The last inequality holds whenever $d/n \le C/\log n$; if $C$ is sufficiently large, this always holds
in the range $1 \le d \le n^{0.9}$. Rename $2C$ to $C$ to complete the proof 
of the lower bound in Corollary~\ref{cor: main}.

\subsection{Upper bound: five cases}
%------------------

As we are about to show, the upper bound in Corollary~\ref{cor: main} 
holds in the entire range $n > 1$, $1 \le d \le n$. 
The argument will again follow from Theorem~\ref{thm: main} together with  
fine estimates of binomial sums. The desired inequality is exact 
and holds even for small values of $d$ and $n$. The argument is elementary but somewhat long since several different cases must be considered.

Our reasoning will be different depending on whether the numbers $n$ and $d$ 
are small or large, and also on 
the size of the ratio $n/d$. These distinctions break the argument into five cases, 
which are summarized in Figure~\ref{fig: cases} below. 
As we see from the figure, the critical ratio for $n/d$ in the argument is given by: 
\begin{equation}	\label{eq: ratio}
\a = 3.0528.
\end{equation}

\begin{figure}[htp]		
  \includegraphics[width=.35\textwidth]{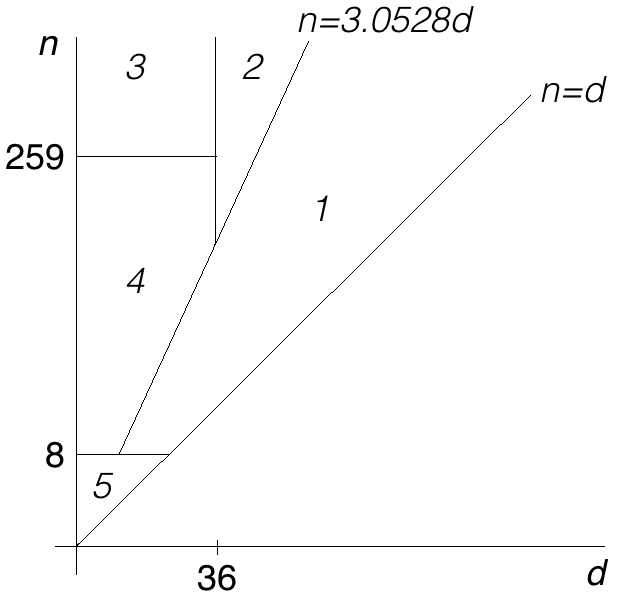} 
  \caption{The proof has five cases, which are determined by the five regions the points $(n,d)$ belong to. 
    In region~1, the bound is trivial since it exceeds the number of all Boolean functions. 
    In regions~2 and 3, the result is derived using fine approximations of binomial sums. 
    In the finite regions~4 and 5, the bound is tested numerically by computer for all pairs $(n,d)$ in those regions.}
  \label{fig: cases}
\end{figure}

\subsection{Case 1: $8 \le n \le \a d$}
%----------------------

We are going to check that in this range:
\begin{equation}	\label{eq: bigger than everything}
\frac{n^{d+1}}{d!} > 2^n.
\end{equation}
This would trivially yield the upper bound in Corollary~\ref{cor: main},
since $T(n,d)$ is bounded by the total number of Boolean functions $2^{2^n}$. And this would yield:
$$
\log_2 T(n,d) \le 2^n < \frac{n^{d+1}}{d!}.
$$

In order to show \eqref{eq: bigger than everything}, Stirling's approximation can be used to check  
that the bound: 
\begin{equation}	\label{eq: Stirling}
d! \le (d/e)^d \sqrt{d} \, e
\end{equation}
holds for all positive integers $d$, and so we have:
$$
\frac{n^{d+1}}{d!} \ge \left( \frac{en}{d} \right)^d \frac{n}{e \sqrt{d}} 
>\left( \frac{en}{d} \right)^d.
$$
In the last step, we used that $\frac{n}{e \sqrt{d}} > 1$ due to the constraints $n \ge 8$ and $d \le n$.
Thus, in order to complete the proof in this case, it suffices to show that
$$
\left( \frac{en}{d} \right)^d \ge 2^n.
$$
Taking logarithms on both sides, we see that this is the same as showing that 
$$
\log_2(ex) \ge x \quad \text{where } x = n/d.
$$

A quick check shows that $\log_2(ex) \ge x$ for all values $x \in [1, \a]$
where $\a = 3.0528$ as in \eqref{eq: ratio}.
It remains to note that the specific value $x = n/d$ falls into the interval $[1, \a]$: 
the lower bound holds since $d \le n$, and the upper 
bound is exactly our assumption that $n \le \alpha d$. 

The proof in Case~1 is complete.

\subsection{Case 2: $n > \a d$, $d \ge 36$}
%----------------------

In this regime, we use Theorem~\ref{thm: upper bound sharp}. To simplify its conclusion,  
let us find a convenient bound on:
$$
m = \binom{n}{\le d}.
$$
Lemma~\ref{lem: binomial sums} gives:
$$
m \le \frac{n^d}{d!} \cdot \frac{n+1-d}{n+1-2d}.
$$
Furthermore, we have:
$$
\frac{n+1-d}{n+1-2d} 
= 1 + \frac{d/n}{1+1/n-2d/n} 
\le 1 + \frac{d/n}{1 - 2/\a}
  %&\le 1 + 2.9 \frac{d}{n}  \quad \text{(by the choice of $\a$ in \eqref{eq: ratio}).}
$$
where the last inequality follows since $1/n \ge 0$ and $d/n \le 1/\alpha$.
Thus:
\begin{equation}	\label{eq: D0}
m \le \left\lfloor \frac{n^d}{d!} \left( 1 + \frac{d/n}{1 - 2/\a} \right) \right\rfloor =: m_0.
\end{equation}

Using Theorem~\ref{thm: upper bound sharp}, we obtain:
$$
T(n,d) 
  < 2 \binom{2^n}{\le m_0}
  \le 2 \left( \frac{e 2^n}{m_0} \right)^{m_0},
$$
where we used Lemma~\ref{lem: binomial bounds} in the last step. 
Thus:
\begin{equation}	\label{eq: logT begin}
\log_2 T(n,d) < 1 + m_0 \left( \log_2 e + n - \log_2 m_0 \right).
\end{equation}

To proceed, we need a convenient lower bound on $\log_2 m_0$. 
The definition of $m_0$ and Stirling's approximation \eqref{eq: Stirling} give:
$$
m_0 
\ge \frac{n^d}{d!}
\ge \left( \frac{en}{d} \right)^d \frac{1}{\sqrt{d} \, e}, 
$$
so: 
\begin{align*} 
\log_2 m_0 
  &\ge d \log_2 \left( \frac{en}{d} \right) - \frac{1}{2} \log_2 d - \log_2 e \\
  &\ge d \log_2 (e \alpha) - \frac{1}{2} \log_2 d - \log_2 e  \quad \text{(since $n > \a d$ by assumption).}
\end{align*}

Substituting this lower bound and the definition \eqref{eq: D0} of $m_0$ into \eqref{eq: logT begin}, yields:
$$
\log_2 T(n,d) 
< 1 + \frac{n^d}{d!} \left( 1 + \frac{d/n}{1 - 2/\a} \right) 
	\left( n - f(d) \right)
$$
where
$$
f(d) := d \log_2 (e \alpha) - \frac{1}{2} \log_2 d - 2 \log_2 e.
$$
Redistributing the powers of $n$, this is the same as: 
$$
\log_2 T(n,d) 
\le 1 + \frac{n^{d+1}}{d!} \left( 1 + \frac{d/n}{1 - 2/\a} \right) \left( 1 - \frac{f(d)}{n} \right).
$$

We claim that:
\begin{equation}	\label{eq: fd}
f(d) \ge \frac{d}{1 - 2/\a} \quad \text{for all } d \ge 36.
\end{equation}
By definition of $f(d)$, we can write this desired inequality as:
\begin{equation}	\label{eq: inequality for d}
\left( \log_2(e\a) - \frac{1}{1 - 2/\a} \right) d - \frac{1}{2} \log_2 d - 2 \log_2 e > 0.
\end{equation}
Since $\a = 3.0528$ by \eqref{eq: ratio} and since $2 \log_2 e \approx 2.88539$, 
this inequality is (slightly) weaker than: 
$$
0.1528 d - 0.5 \log_2 d - 2.8854 \ge 0.
$$
A quick verification shows that this inequality holds for all real values $d \ge 36$, 
so the claim \eqref{eq: fd} is true. 

Using \eqref{eq: fd} we obtain:
\begin{align*} 
\log_2 T(n,d) 
  &< 1 + \frac{n^{d+1}}{d!} \left( 1 + \frac{d/n}{1 - 2/\a} \right) \left( 1 - \frac{d/n}{1 - 2/\a} \right) \\
  &= 1 + \frac{n^{d+1}}{d!} \left(1 - \left( \frac{d/n}{1 - 2/\a} \right)^2 \right) \\
  &= \frac{n^{d+1}}{d!} + 1 - \frac{n^{d+1}}{d!} \left( \frac{d/n}{1 - 2/\a} \right)^2.
\end{align*}
Now, by \eqref{eq: ratio} $0 < 2/\a < 1$, so $(1-2/\a)^2 < 1$. Thus:  
$$
\frac{n^{d+1}}{d!} \left( \frac{d/n}{1 - 2/\a} \right)^2
> \frac{n^{d+1}}{d!} \left(\frac{d}{n} \right)^2
= \frac{n^{d-1}}{(d-1)!} \cdot d \ge 1
$$
since $n \ge d \ge 1$. 
This shows that: 
$$
\log_2 T(n,d) < \frac{n^{d+1}}{d!},
$$
which completes the proof in Case~2.

\subsection{Case 3: $n \ge 259$, $d \le 35$}
%----------------------

In this regime: 
$$
\frac{n}{d} \ge \frac{259}{35} = 7.4.
$$
Thus, we may repeat the argument of Case~2 for the larger, and thus better, value 
$\a = 7.4$. With this value of $\alpha$, the inequality \eqref{eq: inequality for d} is (slightly) weaker than:
$$
2.9598 d - 0.5 \log_2 d - 2.8854 \ge 0.
$$
It is easy to check numerically by computer that this inequality holds for all real values $d \ge 1$. 
The rest of the argument is identical to that of Case~2.

\subsection{Case 4: $8 \le n \le 258$, $d \le 35$, $n > \a d$}
%----------------------

%Let 
%$$
%D = D(n,d) := \binom{n}{\le d}
%$$
Theorem~\ref{thm: upper bound sharp} and Lemma~\ref{lem: binomial bounds}
on the binomial sum
yield:
$$
T(n,d) 
< 2 \binom{2^n-1}{\le m} 
\le 2 \left( \frac{e (2^n-1)}{m} \right)^{m}.
$$
Taking logarithms on both sides gives:
$$
\log_2 T(n,d) < 1 + m \log_2 \left( \frac{e (2^n-1)}{m} \right) =: t(n,d).
$$
It is easy to check numerically by computer that:
%(see the supporting Matlab code \verb=case4.m=) 
$$
t(n,d) \le \frac{n^{d+1}}{d!}
$$
for all pairs $(n,d)$ in the current range.

\subsection{Case 5: $1 < n \le 7$, $1 \le d \le n$}
%----------------------

This is similar to Case~4, except we will not use any bounds on the binomial sum.
As before: 
$$
T(n,d) 
< 2 \binom{2^n-1}{\le m} 
$$
and it is easy to check numerically by computer that:
%(see the supporting Matlab code \verb=case5.m=) 
$$
\log_2 \Big[ 2 \binom{2^n-1}{\le m} \Big] \le \frac{n^{d+1}}{d!} 
$$
for all pairs $(n,d)$ in the current range.

The proof of Corollary~\ref{cor: main} is complete. \qed

\section{Further questions} 		\label{s: further questions}
%=========================

The results and especially the methods of this paper lead to a number of interesting directions 
for further study. 

\subsection{Polynomial capacity of sets.}

Given a finite set $S$ of points in $\R^n$, we can define the capacity $C_d(S)$ to be the base two logarithm of the number of different ways $S$ can be split by polynomials of degree $d$. We can derive some bounds on the polynomial capacity of sets.

\begin{theorem}[Polynomial set capacity]			\label{thm: polynomial set capacity}
  Consider a finite subset $S \subset \R^n$, where $n>1$.
  Then, for any degree $1 < d \le n$, we have:
  $$
  C_d(S) \le 1 + \log_2 \binom{|S|-1}{\le m-1}
  \le m \log_2 \vert S\vert,  
  $$ 
  where:
$$
m=\binom{n+k-1}{\le d}
\le \Big( \frac{2en}{d} \Big)^d.
$$  
\end{theorem}

\begin{proof}
First, it is easy to see that the number of coefficients of a polynomial of degree $d$ in $n$ variables $x_1, \ldots,x_n$ is given by $m=\binom{n+k-1}{\le d}$, including the constant term (bias).
A vector $x \in \R^n$ can be canonically and injectively mapped into a 
vector $f(x) \in \R^{m-1}$ whose components are the various monomials. 
 Using this mapping, we can represent any polynomial $p(x)$ of degree $d$ over $\R^n$
as a linear affine function over $f(x)$. 
And vice versa, any linear affine function over $f(x)$ is a polynomial of degree $d$ over $x$.
For example, if $d=2$, the vector $x=(x_1,x_2) \in \R^2$ is canonically mapped 
to the vector $f(x)=(x_1,x_2, x_1x_2, x_1^2,x_2^2) \in \R^5$.
Any polynomial $p(x) = a_0 + a_1 x_1 + a_2 x_2 + a_{12} x_1 x_2 + a_{11} x_1^2 + a_{22} x_2^2$
over $\R^2$ is clearly an affine function of $f(x)$, and vice versa.
Therefore: 
$$
C_d(S) = C_1(f(S))=C(f(S)).
$$
We complete the proof by applying 
the known bounds on the capacity of sets with respect to linear threshold functions (see \cite{baldi2019capacity})
to the set 
$f(S) \subset \R^{m-1}$, noting that $f(S)$ has the same cardinality as $S$ since $f$ is injective. 
\end{proof}
Note that if we apply Theorem~\ref{thm: polynomial set capacity} to the hypercube $S=H^n$, we get: 
$$ C_d(H^n) \leq n \Big( \frac{2en}{d} \Big)^d,
$$
which is somewhat weaker asymptotically than the result in 
\cite{baldi2018booleanv1} giving: 
$$C_d(H^n)=C_d(n,1)= \frac{ n^{d+1}}{d!}(1+o(1)).$$ Note also that the general lower bound:
$ 1+\log_2 \vert S \vert \leq  C_d(S)$, and its improved version when $S$ is a subset of the Boolean cube (\cite{baldi2019capacity}:
$ \log_2^2 \vert S\vert/16 \leq  C_d(S)$ are trivially satisfied. 
Thus an open research area here is to obtain better estimates of the polynomial capacity of finite sets of points, including subsets of the hypercube.

\subsection{Polynomial threshold functions with high degrees.}
%----------------

In this paper we determined the asymptotic behavior of $T(n,d)$, the number of $n$-variable polynomial threshold functions with bounded or slowly growing degrees $d$. 
It would be interesting to find out what happens if the degree $d$ grows rapidly, 
for example linearly with $n$. 
It is plausible that the upper bound on $T(n,d)$ that we stated 
in Theorem~\ref{thm: upper bound sharp} may be tight,
and the following conjecture mentioned by M.~Anthony \cite{Anthony classification} could hold:

\begin{conjecture}			\label{conj: Tnd sharp}
  The number $T(n,d)$ of $n$-variable polynomial threshold functions of degree $d$ satisfy 
  \begin{equation}	\label{eq: Tnd conjecture}
  T(n,d) = \left( 2-o(1) \right) \binom{2^n-1}{\le m-1}
  \quad \text{where} \quad
  m = \binom{n}{\le d}
  \end{equation}
  for all degrees $1 \le d \le n$ as $n \to \infty$. 
\end{conjecture}

This conjecture might be too strong and it possibly holds only after we take logarithms on both sides of \eqref{eq: Tnd conjecture}.

For bounded or mildly growing degrees $d$, Conjecture~\ref{conj: Tnd sharp} 
easily implies the main result of this paper, Theorem~\ref{thm: main}.
For $d=n/2$, Conjecture~\ref{conj: Tnd sharp} and a careful asymptotic analysis of the bound 
\eqref{conj: Tnd sharp} implies the Wang-Williams conjecture mentioned in the introduction, which states that most Boolean functions can be expressed as polynomial threshold functions of degree $n/2$ (see \cite{Anthony classification}).
Finally, for $d=n$, Conjecture~\ref{conj: Tnd sharp} is trivial. In this case it gives $T(n,n) = 2^{2^n}$, 
which is equivalent to the fact that all Boolean functions are polynomial threshold functions of degree at most $n$. 

\subsection{Polynomial threshold functions with restricted coefficients}
%-----------

In some applications (e.g. discrete synapses in neural networks), 
it is useful to consider polynomial threshold functions $f(x) = \sgn(p(x))$ 
where $p(x)$ is required to have bounded, discrete, or positive coefficients. 

\subsubsection{Integer coefficients}

By an easy perturbation argument, we can always force $p(x)$ to have integer coefficients. 
How large are these coefficients? If $d=1$, i.e. in the case of linear threshold functions, all  
coefficients of $p(x)$ are bounded by $n^{n/2 + o(n)}$. This bound is tight due to results of 
J. H\r{a}stad \cite{Hastad} and N.~Alon and V.~Vu \cite{Alon-Vu}. 
For any higher degree $d \ge 2$, 
a similar result is described in 
\cite{Podolskii}.

Beyond these results, it may be natural to look for a bound on the coefficients of $p(x)$
that holds for {\em most} (or many) polynomial threshold functions $f(x) = \sgn(p(x))$, 
e.g. for $2^{(1+o(1)) n^{d+1}/d! }$ of them. What is then the optimal bound? Is it significantly 
smaller than $n^{n/2 + o(n)}$, the bound that holds for {\em all} functions $f(x)$?
This question seems to be open for all degrees including $d=1$.

A related problem is when we require the integer coefficients of $p(x)$ to be bounded 
by a given number $M$. How many polynomial threshold functions $f(x) = \sgn(p(x))$ can be generated with this restriction? What if we consider polynomials $p(x)$ with all $\pm 1$ coefficients? 
Since each such polynomial consists of $\binom{n}{\le d}$ monomial terms and each term is assigned an $\pm 1$ coefficient, 
there are at most $2^{\binom{n}{\le d}}$ polynomials with $\pm 1$ coefficients. Thus the number of corresponding polynomial threshold functions is bounded by $2^{\binom{n}{\le d}}$. Is this bound asymptotically tight? It is easy to check that the answer is positive for $d=1$, but for higher degrees the problem is non-trivial.

\subsubsection{Positive coefficients}

In some other situations (e.g. excitatory neurons in neural networks), it is natural to consider 
polynomial threshold functions $f(x) = \sgn(p(x))$ where the polynomials $p(x)$ have positive coefficients. 
How many polynomial threshold functions can be generated with this restriction? 

For $d=1$, one can answer this question easily by leveraging the symmetry of the Boolean cube $\{-1,1\}^n$ with respect to signs. Due to this symmetry, the number of homogeneous linear threshold functions $f(x) = \sgn(a_1 x_1 + \cdots a_n x_n)$ whose coefficients $a_k$ follow a given sign pattern is the same for each pattern. It follows that 
$$
\bar{T}^+(n,1) = \frac{\bar{T}(n,1)}{2^n}
$$
where $\bar{T}(n,1)$ denotes the number of homogeneous linear threshold functions, 
and $\bar{T}^+(n,1)$ denotes the number of such functions with positive coefficients. 
Since $\log_2 \bar{T}(n,1) = n^2 - o(n^2)$, we get 
$$
\log_2 \bar{T}^+(n,1) = n^2 - o(n^2) - n = n^2 - o(n^2).
$$
However, for higher degrees $d \ge 2$, the symmetry argument fails and the problem remains open. 

\subsection{Connections to information theory.}
%----------------

The main result of this paper, Theorem~\ref{thm: main}, implies that we need essentially 
$n^{d+1}/d!$ bits to communicate a polynomial threshold function of degree $d$.
This can be viewed as $n^{d}/d!$ binary vectors of dimension $n$ and can intuitively be understood as communicating $n^{d}/d!$ 
{\it support} vectors, that is the $n^{d}/d!$ vectors of the Boolean hypercube that are closest to and on one side of the corresponding separating surface $p(x)=0$. Thus in this case, the set of vectors depends on the function and thus it needs to be be communicated. 
However we may consider fixing a set of vectors in advance -- one set for any function being communicated.  In this scheme, we need to send only the value of the function $f(x)$ on this set.
How well will such a scheme work? How large does the set of vectors needs to be for exact or 
approximate communication of any (or most) polynomial threshold functions of a given degree $d$?

\subsection{Boolean networks.}
%----------------

In neural networks and other applications, one is interested in the behavior of entire networks (or circuits) of polynomial threshold functions, rather than single polynomial threshold functions. For a given circuit, one would like to estimate the number of different Boolean functions that can be realized using different weights. This question has seemed hopeless for a long time but we believe the results presented here can be used to make some progress, at least in the case of particular circuits 
that are widely used in applications. Results in this direction are described in \cite{baldi2018neuronal,baldi2019capacity} where we show how to estimate the number of functions that can be computed by fully connected networks, as well as shallow and deep layered feedforward networks of polynomial threshold gates.

\subsection{The geometry of boolean threshold functions}
%----------------

As we noted in Section~\ref{s: hyperplane arrangements}, homogeneous linear threshold functions correspond to the regions of the hyperplane arrangement $x^\perp$, $x \in \{-1,1\}^n$. 
These regions are polyhedral cones in $\R^n$, and to study their geometry it is convenient to intersect them with
the unit Euclidean sphere. Thus we are looking at a decomposition of the sphere by $2^n$ central hyperplanes. From \eqref{eq: Zuev} we know that there are approximately $2^{n^2}$ regions in this decomposition. What else do we know about them? For example, what is the distribution of their area?
We can of course ask the same questions for $d>1$ as well.

These problems are related to the classical study of random Poission {\em tessellations} in stochastic geometry \cite{Calka}; see also \cite{Plan-Vershynin} for random tessellations on the sphere. 
However, the main new challenge here is to handle the discrete distribution induced by the Boolean cube.

\appendix
\section{Bounds on binomial sums}
%-----------

Throughout the main body of this paper, we repeatedly used the following 
elementary and well known bounds on binomial sums.

\begin{lemma}[see e.g. Exercise~0.0.5 in \cite{V book}]	\label{lem: binomial bounds}
  For any integers $1 \le d \le n$, we have:
  $$
  \Big( \frac{n}{d} \Big)^d 
  \le \binom{n}{d} 
  \le \binom{n}{ \le d} 
  \le \Big( \frac{en}{d} \Big)^d.
  $$
\end{lemma}

\begin{lemma}		\label{lem: more binomial bounds}
  For any integers $d$ and $n$ such that $1 \le d \le n/2$ and for any integer 
  $1 \le k \le n-d+1$, we have:
  $$
  \Big( 1-\frac{2k}{n} \Big)^d \binom{n}{\le d} \le \binom{n-k}{\le d} \le \binom{n}{\le d}.
  $$
\end{lemma}

\begin{proof}
Let us first prove a similar but simpler fact for binomial coefficients: 
\begin{equation}	\label{eq: binom defect} 
\Big( 1-\frac{2k}{n} \Big)^d \binom{n}{d} \le \binom{n-k}{d} \le \binom{n}{d} 	
\end{equation}
The upper bound is non-trivial. To prove the lower bound,
combine the terms in the definition of the binomial coefficients and get:
$$
\frac{\binom{n-k}{d}}{\binom{n}{d}}
= \prod_{i=0}^{d-1} \Big( 1 - \frac{k}{n-i} \Big)
\ge \Big( 1 - \frac{k}{n-d+1} \Big)^d.
$$
Since $n-d+1 \ge n/2$, the lower bound in \eqref{eq: binom defect} follows.

Next, we can deduce the conclusion of the lemma from \eqref{eq: binom defect} as follows:
$$
\sum_{i=1}^d \binom{n-k}{i}
\ge \sum_{i=1}^d \Big( 1-\frac{2k}{n} \Big)^i \binom{n}{i}
\ge \Big( 1-\frac{2k}{n} \Big)^d \sum_{i=1}^d \binom{n}{i},
$$
which is the lower bound we claimed.
\end{proof}

\begin{lemma}	\label{lem: binomial sums}
  For any integers $1 \le d \le n/2$, we have:
  $$
  \binom{n}{\le d} \le \binom{n}{d} \frac{n+1-d}{n+1-2d}.
  $$
\end{lemma}

\begin{proof}
Let us bound the ratio
\begin{align*}
\frac{\binom{n}{\le d}}{\binom{n}{d}} 
  &= \frac{\binom{n}{d} + \binom{n}{d-1} + \binom{n}{d-2} + \cdots}{\binom{n}{d}} \\
  &= 1 + \frac{d}{n-d+1} + \left( \frac{d}{n-d+1} \right) \left( \frac{d-1}{n-d+2} \right) + \cdots \\
  &\le 1 + \frac{d}{n-d+1} + \left( \frac{d}{n-d+1} \right)^2 + \cdots \\
  &= \frac{n+1-d}{n+1-2d} \quad \text{(by summing the geometric series).}
\end{align*}
This proves the lemma.
\end{proof}


\begin{thebibliography}{99}

\bibitem{ASW} E.~Abbe, A.~Shpilka, and A.~Wigderson,
{\em Reed–Muller codes for random erasures and errors,} 
IEEE Transactions on Information Theory 61 (2015), 5229--5252.

\bibitem{Adamczak-et-al} R. Adamczak, O. Gu\'edon, A. Litvak, A. Pajor, N. Tomczak-Jaegermann, 
{\em Smallest singular value of random matrices with independent columns,} 
C. R. Math. Acad. Sci. Paris 346 (2008), 853--856.

\bibitem{ACW} J. Alman, T. Chan, R. Williams,
{\em Polynomial representations of threshold functions and algorithmic applications,}
57th Annual IEEE Symposium on Foundations of Computer Science (FOCS 2016).

\bibitem{Alon-Vu} N. Alon, V. Vu, 
{\em Anti-Hadamard matrices, coin weighing, threshold gates, and indecomposable hypergraphs,} 
J. Combin. Theory Ser. A 79 (1997), 133--160.

\bibitem{Anthony classification} M.~Anthony, 
{\em Classification by polynomial surfaces,}
Discrete Applied Mathematics 61 (1995), 91--103.

\bibitem{Anthony} M. Anthony, 
{\em Discrete mathematics of neural networks. Selected topics.} 
SIAM Monographs on Discrete Mathematics and Applications. 
Society for Industrial and Applied Mathematics (SIAM), Philadelphia, PA, 2001. 

\bibitem{ABFR} J. Aspnes, R. Beigel, M. Furst, S. Rudich, 
{\em The expressive power of voting polynomials,} 
Combinatorica, 14 (1994), 1--14.

\bibitem{baldi87} P. Baldi, 
{\em Symmetries and learning in neural network models,}
Phys. Rev. Lett. 59 (1987), no. 17, 1976--1978.

\bibitem{Baldi} P. Baldi, 
{\em Neural networks, orientations of the hypercube, and algebraic threshold functions,} 
IEEE Trans. Inform. Theory 34 (1988), no. 3, 523--530. 

\bibitem{baldi88b} P. Baldi, 
{\em Group actions and learning for a family of automata,}
J. Comput. System Sci. 36 (1988), no. 1, 1--15.

\bibitem{baldi2018booleanv1}
Pierre Baldi and Roman Vershynin.
\newblock {\em Boolean polynomial threshold functions and random tensors.}
\newblock arXiv preprint arXiv:1803.10868, v1, 2018.

\bibitem{baldireview2018}
P.~Baldi.
\newblock {\em Deep learning in biomedical data science.}
\newblock Annual Review of Biomedical Data Science, 1 (1988),181--205.

\bibitem{baldi2018neuronal}
P.~Baldi and R.~Vershynin.
\newblock {\em On neuronal capacity.}
\newblock Advances in Neural Information Processing Systems, (2018), 7740--7749.

\bibitem{baldi2019capacity}
P.~Baldi and R.~Vershynin.
\newblock {\em The capacity of feedforward neural networks.}
\newblock Neural Networks, 116, (2019), 288-311. See also
arXiv preprint: arXiv:1901.00434, 2018.

\bibitem{Basak-Rudelson}  A. Basak, M. Rudelson, 
{\em Invertibility of sparse non-Hermitian matrices,} 
Adv. Math. 310 (2017), 426--483.

\bibitem{Beigel} R. Beigel,
{\em The polynomial method in circuit complexity,} 
Proc. of 8th Annual Structure in Complexity Theory Conference (1993), 82--95.

\bibitem{BGS} A. Bhattacharyya, S. Ghoshal, R. Saket,
{\em Hardness of learning noisy halfspaces using polynomial thresholds,}
preprint (2017). 

\bibitem{BRS} R. Beigel, N. Reingold, D. Spielman,
{\em PP is closed under intersection,} 
Journal of Computer \& System Sciences 50 (1995), 191--202.

\bibitem{Bollobas} B. Bollob\'as, 
{\em Combinatorics. Set systems, hypergraphs, families of vectors and combinatorial probability.}
Cambridge University Press, Cambridge, 1986.

\bibitem{Bourgain-Wood-Vu} J. Bourgain, V. Vu, P. Wood, 
{\em On the singularity probability of discrete random matrices,}
J. Funct. Anal. 258 (2010), 559--603.

\bibitem{Bruck} J. Bruck,
{\em Harmonic analysis of polynomial threshold functions,}
SIAM J. Discrete Math. 3 (1990), 168--177. 

\bibitem{Buck} R. C. Buck, 
{\em Partition of space,}
Amer. Math. Monthly 50 (1943), 541--544. 

\bibitem{Calka} P. Calka, 
{\em Tessellations.} 
New perspectives in stochastic geometry, 145--169, Oxford Univ. Press, Oxford, 2010.

\bibitem{Cook} N. Cook,
{\em On the singularity of adjacency matrices for random regular digraphs,}
Probab. Theory Related Fields 167 (2017), 143--200. 

\bibitem{Costello} K. Costello, 
{\em Bilinear and quadratic variants on the Littlewood-Offord problem,}
Israel J. Math. 194 (2013), 359--394. 

\bibitem{Costello-Tao-Vu} K. Costello, T. Tao, V. Vu,
{\em Random symmetric matrices are almost surely nonsingular,} 
Duke Math. J. 135 (2006), 395--413.

\bibitem{Comon et al} P. Comon, G. Golub, L.-H. Lim, B. Mourrain,
{\em Symmetric tensors and symmetric tensor rank,} 
SIAM J. Matrix Anal. Appl. 30 (2008), 1254--1279. 

\bibitem{cover1965geometrical} T. Cover,
{\em Geometrical and statistical properties of systems of linear inequalities with applications in pattern recognition,}
IEEE Transactions on Electronic Computers 3 (1965), 326--334.

\bibitem{delaPena-Gine} V. de la Pe\~na, E.~Gin\'e,
{\em Decoupling: from dependence to independence.} 
Springer Verlag, 1999.

\bibitem{DOSW} I. Diakonikolas, R. O'Donnell, R. Servedio, Y. Wu, 
{\em Hardness results for agnostically learning low-degree polynomial threshold functions.} 
Proceedings of the Twenty-Second Annual ACM-SIAM Symposium on Discrete Algorithms, 
1590--1606, SIAM, Philadelphia, PA, 2011.

\bibitem{DSTW} I. Diakonikolas, R. A. Servedio, L.-Y. Tan, A. Wan,
{\em A regularity lemma and low-weight approximators for low-degree polynomial threshold functions,} Theory Comput. 10 (2014), 27--53. 

\bibitem{Erdos} P. Erd\"os, 
{\em On a lemma of Littlewood and Offord,} 
Bull. Amer. Math. Soc. 51 (1945), 898--902.

\bibitem{Erdos-Moser} P. Erd\"os, 
{\em Extremal problems in number theory,} 
Proc. Sympos. Pure Math., vol.VIII, AMS, Providence, RI, 1965, pp.181--189.

\bibitem{Frankl-Furedi} P. Frankl,Z. F\"uredi, 
{\em Solution of the Littlewood-Offord problem in high dimensions,} 
Ann. of Math. (2) 128 (1988), 259--270. 

\bibitem{Gotze} F. G\"otze, 
{\em Asymptotic expansions for bivariate von Mises functionals,} 
Z. Wahrsch. Verw. Gebiete 50 (1979), 333--355.

\bibitem{GNT} F. G\"otze, A. Naumov, A. Tikhomirov, 
{\em On minimal singular values of random matrices with correlated entries,} 
Random Matrices Theory Appl. 4 (2015), no. 2, 1550006.

\bibitem{Halasz} G. Hal\'asz, 
{\em Estimates for the concentration function of combinatorial number theory and probability,} 
Period. Math. Hungar. 8 (1977) 197--211.

\bibitem{Hastad} J. H\r{a}stad, 
{\em On the size of weights for threshold gates,} 
SIAM J. Discrete Math. 7 (1994), 484--492.

\bibitem{Irmatov} A. A. Irmatov, 
{\em On the number of threshold functions,}
Diskret. Mat. 5 (1993), 40--43; translation in 
Discrete Math. Appl. 3 (1993), 429--432.

\bibitem{Irmatov2} A. A. Irmatov, 
{\em Arrangements of hyperplanes and the number of threshold functions,}
Acta Appl. Math. 68 (2001), 211--226. 

\bibitem{KKS} J. Kahn, J. Koml\'os, E. Szemer\'edi, 
{\em On the probability that a random $\pm 1$-matrix is singular,}
J. Amer. Math. Soc. 8 (1995), 223--240.

\bibitem{Kalton} N. Kalton, 
{\em Rademacher series and decoupling,} 
New York J. Math. 11 (2005), 563--595.

\bibitem{Kane} D. Kane, 
{\em A structure theorem for poorly anticoncentrated polynomials of Gaussians and applications to the study of polynomial threshold functions,} 
Ann. Probab. 45 (2017),  1612--1679. 

\bibitem{KV} Z. Kovijani\'c Vuki\'cevi\'c, 
{\em An enumerative problem in threshold logic,}
Publ. Inst. Math. (Beograd) (N.S.) 82(96) (2007), 129--134. 

\bibitem{Kannan} R. Kannan, 
{\em Decoupling and partial independence,} 
Building bridges, 321--331, Bolyai Soc. Math. Stud., 19, Springer, Berlin, 2008.

\bibitem{Klivans-ODonnell-Servedio} A. Klivans, R. O'Donnell, R. Servedio,
{\em Learning intersections and thresholds of halfspaces.} 
In Proceedings of the 43rd Annual Symposium on Foundations of Computer Science (2002), 177--186.

\bibitem{Klivans-Servedio} A. Klivans, R. Servedio,
{\em Learning DNF in time $2^{O(n^{1/3})}$,} 
J. Computer and System Sciences 68 (2004), 303--318.

\bibitem{kolda2009tensor} T. Kolda, B. Bader, 
{\em Tensor decompositions and applications,} 
SIAM Rev. 51 (2009), no. 3, 455-?500.

\bibitem{Komlos1} J. Koml\'os, 
{\em On the determinant of $(0,1)$ matrices,}
Studia Sci. Math. Hungar 2 (1967), 7--21.

\bibitem{Komlos2} J. Koml\'os, 
{\em On the determinant of random matrices,} 
Studia Sci. Math. Hungar. 3 (1968), 387--399. 

\bibitem{Krause-Pudlak} M. Krause, P. Pudlak, 
{\em Computing boolean functions by polynomials and threshold circuits,} 
Computational Complexity 7 (1998), 346--370.

\bibitem{Ledoux} M. Ledoux, 
{\em The concentration of measure phenomenon.} 
Mathematical Surveys and Monographs, 89. American Mathematical Society, Providence, RI, 2001.

\bibitem{Littlewood-Offord} J. E. Littlewood, A. C. Offord, 
{\em On the number of real roots of a random algebraic equation. III,}
Rec. Math. [Mat. Sbornik] N.S. 12 (1943), 277--286;
in {\em Collected Papers of J.~E.~Littlewood},
Vol. 2, pp. 1333--1342, Oxford University Press, London, 1982.

\bibitem{Litvak-et-al}  A. Litvak, A. Lytova, K. Tikhomirov, N. Tomczak-Jaegermann, P. Youssef, 
{\em Adjacency matrices of random digraphs: singularity and anti-concentration,} 
J. Math. Anal. Appl. 445 (2017), 1447--1491.

\bibitem{Matousek} J. Matou\v{s}ek,
{\em Lectures on discrete geometry.}
Graduate Texts in Mathematics, 212. Springer-Verlag, New York, 2002. 

\bibitem{mcculloch:43} W. McCulloch, W. Pitts, 
{\em A logical calculus of the ideas immanent in nervous activity,} 
Bull. Math. Biophys. 5 (1943), 115--133.

\bibitem{mccullagh1987tensor} P. McCullagh,
{\em Tensor methods in statistics.}
Monographs on Statistics and Applied Probability. Chapman \& Hall, London, 1987. 

\bibitem{MNV} R. Meka, O. Nguyen, V. Vu,
{\em Anti-concentration for polynomials of independent random variables,} 
Theory Comput. 12 (2016), Paper No. 11, 16 pp.
  
\bibitem{Minsky-Papert} M. Minsky, S. Papert,
{\em Perceptrons: an introduction to computational geometry}
(expanded edition); MIT Press, Cambridge, MA, 1988.

\bibitem{muroga1965lower} S. Muroga, 
{\em Lower bounds of the number of threshold functions and a maximum weight,}
IEEE Transactions on Electronic Computers 2 (1965), 136--148.

\bibitem{Nering} E. Nering, 
{\em Linear Algebra and Matrix Theory.} 
Second edition. New York: Wiley, 1970.

\bibitem{Nguyen EJP} H. Nguyen, 
{\em On the least singular value of random symmetric matrices,}
Electron. J. Probab. 17 (2012), no. 53, 19 pp. 

\bibitem{Nguyen Duke} H. Nguyen, 
{\em Inverse Littlewood-Offord problems and the singularity of random symmetric matrices,} 
Duke Math. J. 161 (2012), 545--586. 

\bibitem{Nguyen SIAM} H. Nguyen, 
{\em On the singularity of random combinatorial matrices,} 
SIAM J. Discrete Math. 27 (2013), 447--458. 

\bibitem{Nguyen-Vu} H. Nguyen, V. Vu,
{\em Optimal inverse Littlewood-Offord theorems,} 
Adv. Math. 226 (2011), 5298--5319.

\bibitem{Odlyzko} A. M. Odlyzko, 
{\em On subspaces spanned by random selections of $\pm 1$ vectors,}
Journal of Combinatorial Theory, Series A 47 (1988), 124--133.

\bibitem{ODonnell book} R. O'Donnell, 
{\em Analysis of Boolean functions.} 
Cambridge University Press, New York, 2014. 

\bibitem{OS1} R. O'Donnell, R. A. Servedio, 
{\em Extremal properties of polynomial threshold functions,} 
J. Comput. System Sci. 74 (2008), 298--312.

\bibitem{OS2} R. O'Donnell, R. A. Servedio, 
{\em New degree bounds for polynomial threshold functions,} 
Combinatorica 30 (2010), 327--358.

\bibitem{OZ} R. O'Donnell, Y. Zhao, 
{\em Polynomial bounds for decoupling, with applications,} 
31st Conference on Computational Complexity, Art. No. 24, 18 pp., 
LIPIcs. Leibniz Int. Proc. Inform., 50, Schloss Dagstuhl. 
Leibniz-Zent. Inform., Wadern, 2016. 

\bibitem{Ojha} P. C. Ojha, 
{\em Enumeration of linear threshold functions from the lattice of hyperplane intersections,} IEEE Trans. Neural Networks 11 (2000), 839--850.

\bibitem{Plan-Vershynin} Y. Plan, R. Vershynin, 
{\em Dimension reduction by random hyperplane tessellations,} 
Discrete and Computational Geometry 51 (2014), 438--461.

\bibitem{Podolskii} V.V. Podolskii,
{\em Perceptrons of large weight,}
 Problems of Information Transmission 45, 1, (2009), 46--53.

\bibitem{Rota} G.-C. Rota, 
{\em On the foundations of combinatorial theory. I. Theory of Möbius functions,} 
Z. Wahrscheinlichkeitstheorie und Verw. Gebiete 2 (1964), 340--368.

\bibitem{Rudelson Annals} M. Rudelson, 
{\em Invertibility of random matrices: norm of the inverse,}
Ann. of Math. (2) 168 (2008), 575--600.

\bibitem{RV square} M. Rudelson, R. Vershynin, 
{\em The Littlewood-Offord Problem and invertibility of random matrices,} 
Advances in Mathematics 218 (2008), 600--633.

\bibitem{RV upper} M. Rudelson, R. Vershynin, 
{\em The least singular value of a random square matrix is $O(n^{-1/2})$,} 
C. R. Math. Acad. Sci. Paris 346 (2008), 893--896.

\bibitem{RV rectangular} M. Rudelson, R. Vershynin, 
{\em Smallest singular value of a random rectangular matrix,} 
Communications on Pure and Applied Mathematics 62 (2009), 1707--1739.

\bibitem{RV ICM} M. Rudelson, R. Vershynin,
{\em Non-asymptotic theory of random matrices: extreme singular values.}
Proceedings of the International Congress of Mathematicians. Volume III, 1576--1602,
Hindustan Book Agency, New Delhi, 2010.

\bibitem{RV Hanson-Wright} M. Rudelson, R. Vershynin, 
{\em Hanson-Wright inequality and sub-gaussian concentration,}
Electronic Communications in Probability 18 (2013), 1--9.

\bibitem{RV JAMS} M. Rudelson, R. Vershynin, 
{\em Invertibility of random matrices: unitary and orthogonal perturbations,} 
Journal of the AMS 27 (2014), 293--338.

\bibitem{RV no-gaps} M. Rudelson, R. Vershynin, 
{\em No-gaps delocalization for general random matrices,} 
Geometric and Functional Analysis 26 (2016), 1716--1776.

\bibitem{Saks} M. Saks, 
{\em Slicing the hypercube.} 
Surveys in combinatorics, 1993 (Keele), 211--255, 
London Math. Soc. Lecture Note Ser., 187, Cambridge Univ. Press, Cambridge, 1993. 

\bibitem{Sarkozy-Szemeredi} A. S\'ark\"ozy, E. Szem\'eredi, 
{\"Uber ein Problem von Erd\"os und Moser,} 
Acta Arith. 11 (1965), 205--208.

\bibitem{schmidhuber2015deep} J. Schmidhuber, 
{\em Deep learning in neural networks: An overview,}
Neural Networks 61 (2015), 85--117.

\bibitem{Schlafli} L.~Schl\"afli, 
{\em Gesammelte mathematische Abhandlungen.} 
Band I. (German) Verlag Birkh\"auser, Basel, 1950.

\bibitem{Sherstov} A. Sherstov,
{\em Separating AC0 from depth-2 majority circuits,} 
SIAM J. Computing 38 (2009), 2113--2129.

\bibitem{Stanley}  R. Stanley,
{\em An introduction to hyperplane arrangements.} 
Geometric combinatorics, 389--496, IAS/Park City Math. Ser., 13, 
Amer. Math. Soc., Providence, RI, 2007. 

\bibitem{Talagrand} M. Talagrand, 
{\em A new look at independence,} 
Ann. Probab. 24 (1996), 1--34.

\bibitem{Tao-Vu RSA} T. Tao, V. Vu, 
{\em On random $\pm 1$ matrices: singularity and determinant,} 
Random Structures and Algorithms 28 (2006), 1--23.

\bibitem{Tao-Vu JAMS} T. Tao, V. Vu, 
{\em On the singularity probability of random Bernoulli matrices,} 
J. Amer. Math. Soc. 20 (2007), 603--628.

\bibitem{Tao-Vu Bulletin} T. Tao, V. Vu, 
{\em From the Littlewood-Offord problem to the circular law: universality of the spectral distribution of random matrices,} 
Bull. Amer. Math. Soc. (N.S.) 46 (2009), 377--396.

\bibitem{Tao-Vu GAFA} T. Tao, V. Vu,
{\em Random matrices: the distribution of the smallest singular values,} 
Geom. Funct. Anal. 20 (2010), 260--297. 

\bibitem{Tao-Vu sharp} T. Tao, V. Vu, 
{\em A sharp inverse Littlewood-Offord theorem,}
Random Structures Algorithms 37 (2010), 525--539. 

\bibitem{Tao-Vu Annals} T. Tao, V. Vu, 
{\em Inverse Littlewood-Offord theorems and the condition number of random discrete matrices,} 
Annals of Math. 169 (2009), 595--632.

\bibitem{Tikhomirov Advances} K. Tikhomirov, 
{\em The limit of the smallest singular value of random matrices with i.i.d. entries,} 
Adv. Math. 284 (2015), 1--20. 

\bibitem{Tikhomirov Israel} K. Tikhomirov, 
{\em The smallest singular value of random rectangular matrices with no moment assumptions on entries,}
Israel J. Math. 212 (2016), 289--314.

\bibitem{Tikhomirov IMRN} K. Tikhomirov, 
{\em Sample covariance matrices of heavy-tailed distributions,} 
Int. Math. Res. Notes, to appear. 

\bibitem{Tikhomirov singularity} K. Tikhomirov, 
{\em Singularity of random Bernoulli matrices,}
preprint. 

\bibitem{V symmetric} R. Vershynin,
{\em Invertibility of symmetric random matrices,} 
Random Structures and Algorithms 44 (2014), 135--182.

\bibitem{V book} R. Vershynin, 
{\em High-dimensional probability. An introduction with applications in data science.}
Cambridge University Press, 2017.

\bibitem{V random tensors} R. Vershynin, 
{\em Invertibility of random tensors,}
in preparation.

\bibitem{Voigt-Ziegler} T. Voigt, G. Ziegler,
{\em Singular 0/1-matrices, and the hyperplanes spanned by random 0/1-vectors,}
Combin. Probab. Comput. 15 (2006), no. 3, 463--471.

\bibitem{Wang-Williams} C. Wang, A. Williams,
{\em The threshold order of a boolean function,} 
Discrete Applied Mathematics, 31 (1991), 51--69.

\bibitem{Winder66} R. O. Winder. 
{\em Partitions of n-space by hyperplanes,} SIAM Journal on Applied Mathematics, 14 (1966), no. 4, 811--818.

\bibitem{Wendel} J. Wendel,
{\em A problem in geometric probability,}
Math. Scand. 11 (1962), 109--111. 

\bibitem{Zaslavsky} T. Zaslavsky, 
{\em Facing up to arrangements: face-count formulas for partitions of space by hyperplanes.}
Mem. Amer. Math. Soc. 1 (1975), issue 1, no. 154, vii+102 pp. 

\bibitem{zhang2001rank} T. Zhang, G. Golub, 
{\em Rank-one approximation to high order tensors,}
SIAM J. Matrix Anal. Appl. 23 (2001), no. 2, 534--550. 

\bibitem{Zuev Doklady} Yu. A. Zuev, 
{\em Asymptotics of the logarithm of the number of Boolean threshold functions.} (Russian) 
Dokl. Akad. Nauk SSSR 306 (1989), 528--530; 
translation in Soviet Math. Dokl. 39 (1989), no. 3, 512--513.

\bibitem{Zuev} Yu. A. Zuev, 
{\em Combinatorial-probability and geometric methods in threshold logic.} (Russian)
Diskret. Mat. 3 (1991), no. 2, 47--57; 
translation in Discrete Math. Appl. 2 (1992), no. 4, 427--438.

\bibitem{Zunic} J. Zunic, 
{\em On encoding and enumerating threshold functions,}
IEEE Transactions on Neural Networks 15 (2004), 261--267.

\end{thebibliography}
\end{document}